\documentclass{amsart}
\usepackage[margin=1in]{geometry}

\usepackage{tikz}
\usepackage{physics}
\usepackage{mathdots}
\usepackage{yhmath}
\usepackage{cancel}
\usepackage{siunitx}
\usepackage{array}
\usepackage{multirow}
\usepackage{wrapfig}
\usepackage[font=small]{caption}
\usepackage[all]{xy}
\usepackage{cjhebrew}
\usepackage{wrapfig}

\definecolor{Moonstone}{HTML}{2AB7CA}
\definecolor{SpringGreen}{HTML}{C6DC67}
\definecolor{Sunset}{HTML}{FAD089}
\definecolor{AfricanViolet}{HTML}{A97ABF}

\usepackage{amscd}
\usepackage{gensymb}
\usepackage{tabularx}
\usepackage{extarrows}
\usepackage{booktabs}
\usetikzlibrary{fadings}
\usetikzlibrary{patterns}
\usetikzlibrary{shadows.blur}
\usetikzlibrary{shapes}

\usepackage{amsmath}
\usepackage{amsthm}
\usepackage{amssymb}
\usepackage{mathtools}
\usepackage{soul}
\usepackage{color}

\usepackage{ytableau}

\newtheorem{thm}{Theorem}[section]

\newtheorem{prop}[thm]{Proposition}
\newtheorem{cor}[thm]{Corollary}

\theoremstyle{definition}
\newtheorem{defn}[thm]{Definition}
\newtheorem{ex}[thm]{Example}
\newtheorem{rem}[thm]{Remark}

\let\oldtocsection=\tocsection

\let\oldtocsubsection=\tocsubsection

\let\oldtocsubsubsection=\tocsubsubsection

\renewcommand{\tocsection}[2]{\hspace{0em}\oldtocsection{#1}{#2}}
\renewcommand{\tocsubsection}[2]{\hspace{3em}\oldtocsubsection{#1}{#2}}
\renewcommand{\tocsubsubsection}[2]{\hspace{2em}\oldtocsubsubsection{#1}{#2}}

\usepackage{shuffle}

\newcommand{\shin}{\text{\</s>}}
\newcommand{\rshin}{\mathfrak{R}\text{\</s>}}
\newcommand{\fshin}{\mathchoice 
{\rotatebox[origin=c]{180}{\shin}}
{\rotatebox[origin=c]{180}{\shin}}
{\rotatebox[origin=c]{180}{\scalebox{.7}{\shin}}}
{\rotatebox[origin=c]{180}{\scalebox{.5}{\shin}}}
}
\newcommand{\bshin}{\mathfrak{R}\fshin}

\newcommand{\rskew}{{/\!\!/}}
\newcommand{\otherskew}{skew-II}

\newcommand{\rperp}{ \mathrel{
\begin{tikzpicture}[x=1pt,y=1pt,yscale=-.25,xscale=.25]
\draw[line width=.40pt]    (10,10) -- (20,10) ;
\draw[line width = .40pt]    (20,20) -- (0,20) ;
\draw[line width =.40pt]    (10,0) -- (10,20) ;
\end{tikzpicture}}}

\title[Extended Schur Functions and Involutions]{Extended Schur Functions and Bases Related by Involutions}
\author[S. Daugherty]{Spencer Daugherty} \thanks{spencer.daugherty@colorado.edu. Department of Mathematics, University of Colorado Boulder, USA}

\usepackage[backend=bibtex, maxbibnames=10, doi=false, url=false, isbn=false]{biblatex}
\begin{document}

\maketitle

\begin{abstract} 
We introduce two new bases of $QSym$, the reverse extended Schur functions and the row-strict reverse extended Schur functions, as well as their duals in \textit{NSym}, the reverse shin functions and row-strict reverse shin functions. These bases are the images of the extended Schur basis and shin basis under the involutions $\rho$ and $\omega$, which generalize the classical involution $\omega$ on the symmetric functions.  In addition, we prove a Jacobi-Trudi rule for certain shin functions using creation operators. We define skew extended Schur functions and skew-II extended Schur functions based on left and right actions of \textit{NSym} and $QSym$ respectively. We then use the involutions $\rho$ and $\omega$ to translate these and other known results to our reverse and row-strict reverse bases.\\

\noindent K\textsc{eywords.} Quasisymmetric, Schur-like, NSym, Hopf algebra, tableaux.
\end{abstract}

\section{Introduction}

There has been considerable interest over the last decade in studying Schur-like bases of the noncommutative symmetric functions (\textit{NSym}) and the quasisymmetric functions ($QSym$). A basis $\{\mathcal{S}_{\alpha}\}_{\alpha}$ of \textit{NSym} is generally considered to be \emph{Schur-like} if the commutative image of $\mathcal{S}_{\lambda}$ is $s_{\lambda}$ for any partition $\lambda$. A basis of $QSym$ is Schur-like if it is dual to a Schur-like basis of \textit{NSym}.  These bases are usually defined combinatorially in terms of tableaux that resemble or generalize the semistandard Young tableaux. There are three main pairs of Schur-like bases in \textit{NSym} and $QSym$. In \textit{NSym}, these are the immaculate basis \cite{berg18}, the shin basis \cite{Camp}, and the Young noncommutative Schur basis \cite{Luoto2013}. Their duals in $QSym$ are the dual immaculate basis \cite{berg18},  the extended Schur basis \cite{Assa}, and the Young quasisymmetric Schur basis \cite{Luoto2013}, respectively.

The shin and extended Schur functions, which are dual bases, are unique among the Schur-like bases for having arguably the most natural relationship with the Schur functions. In \textit{NSym}, the commutative image of a shin function indexed by a partition is the Schur function indexed by that partition, while the commutative image of any other shin function is 0. In $QSym$, the extended Schur function indexed by a partition is equal to the Schur function indexed by that partition \cite{Camp}. The goal of this paper is to further the study of this pair of dual bases and to introduce two new pairs of related bases.

In Section \ref{section_creation}, we define a creation operator that can be used to construct a significant number of shin functions. Using this creation operator, we state a Jacobi-Trudi formula for shin functions indexed by strictly increasing compositions. This rule expresses these shin functions as a matrix determinant in terms of the complete homogeneous noncommutative symmetric functions. In Section \ref{section_skew}, we define the skew extended Schur functions algebraically and in terms of skew shin-tableaux, and note their implications for the multiplicative structure of the shin basis. We also define \otherskew\ extended Schur functions using a left action of \textit{NSym} on $QSym$. The skew-II functions are defined combinatorially by instead skewing out the bottom lefthand corner of tableaux. Skew-II functions relate to the product and coproduct similar to the skew functions and become relevant in the following section.

In Section \ref{section_invol}, we introduce two new bases found by applying the $\rho$ and $\omega$ involutions to the extended Schur functions, as well as their dual bases in \textit{NSym}.  We call these new bases the reverse extended Schur functions and the row-strict reverse extended Schur functions, while their duals in \textit{NSym} are the reverse and row-strict reverse shin functions. They are defined combinatorially with tableaux that resemble the shin-tableaux but have different conditions on whether rows increase or decrease and whether that change is weak or strict.  We also connect these bases to the row-strict extended Schur and shin functions which result from the application of the involution $\psi$  to the original bases \cite{row0hecke}. The involutions $\psi$, $\rho$, and $\omega$ on $QSym$ and \textit{NSym} collectively generalize the classical involution $\omega$ on the symmetric functions. In $Sym$, the Schur basis is its own image under $\omega$, but in $QSym$ the extended Schur functions are part of a system of 4 bases that is closed under the three involutions $\psi$, $\rho$, and $\omega$.  The existence of this system of 4 bases as an analogue to the behavior of the Schur functions under $\omega$ is an interesting feature of how these bases generalize the Schur functions. 

Additionally, while the combinatorics of these bases are similar, they may have very different applications. For example, the quasisymmetric Schur basis and the Young quasisymmetric Schur basis are related by the involution $\rho$ but the former is much more compatible with Macdonald polynomials while the latter is more useful when working with Schur functions \cite{haglund2011, Luoto2013}. The quasisymmetric Schur functions are another Schur-like basis that exists in a system of 4 related bases, although each was defined independently and not presented together in this manner \cite{MasNie, MasRem}. Mason and Searles study various polynomials related by $\omega$ in \cite{Mason2021Dichotomy}, using the quasisymmetric Schur functions and the Young quasisymmetric Schur functions as a starting point.

We use the involutions to state row-strict, reverse, and row-strict reverse analogs to our new and classical results on the extended Schur and shin bases. These include a Pieri rule, a rule for multiplication by a ribbon function, a partial Jacobi-Trudi rule, expansions into other bases, a description of the antipode on the shin and extended Schur bases, and two different types of skew functions.  Specifically, we show that the \otherskew\ reverse and row-strict reverse extended Schur functions are the image of the skew extended Schur functions under $\rho$ and $\omega$ respectively.

\subsection{Acknowledgements} The author thanks Laura Colmenarejo and Sarah Mason for their guidance and support,  as well as Kyle Celano and Sheila Sundaram for their thoughtful feedback. We also thank Mike Zabrocki for helpful answers about creation operators and John Campbell for sharing useful Sage code. An extended abstract of this paper is to appear in \textit{Proceedings of the 36th International Conference on Formal Power Series and Algebraic Combinatorics}.

\tableofcontents

\section{Background}

Let $n$ be a positive integer.  A \emph{partition} $\lambda \vdash n$ is a sequence of positive integers $\lambda = (\lambda_1, \ldots, \lambda_k)$ such that $\sum_{i} \lambda_i = n$ and $\lambda_i \geq \lambda_{i+1}$ for all $1 \leq i < k$.  A \emph{composition} $\alpha \models n$ is a sequence  of positive integers $\alpha = (\alpha_1, \ldots, \alpha_k)$ where $\sum_i \alpha_i =n$, so the set of compositions includes all partitions. The \emph{size} of a composition $\alpha = (\alpha_1, \ldots, \alpha_k) \models n $ is $|\alpha|=n$ and the \emph{length} of $\alpha$ is $\ell(\alpha) = k$. The \emph{diagram} of a composition is a left-aligned array of boxes where the $i^{\text{th}}$ row from the top contains $\alpha_i$ boxes. We use the English notation for our diagrams, meaning that the top row is considered to be the first row.  For two compositions $\alpha = (\alpha_1, \ldots, \alpha_k)$ and $\beta = (\beta_1, \ldots, \beta_j)$ with $\beta_i \leq \alpha_i$ for all $1\leq i \leq j$, the \emph{skew shape} $\alpha/\beta$ is a composition diagram of shape $\alpha$ where the first $\beta_i$ boxes in the $i^{\text{th}}$ row are removed for $1 \leq i \leq j$. We represent this removal by shading in the removed boxes.

A \emph{weak composition} of $n$ is a sequence of non-negative integers $\alpha = (\alpha_1, \ldots, \alpha_k)$ where $\sum_i \alpha_i = n$; in other words a composition that allows zeroes. Often we implicitly treat weak compositions as if they have infinitely many zero entries at their end. Given a weak composition $\alpha$, we define the \emph{flattening} of $\alpha$, denoted $\tilde{\alpha}$, to be the composition obtained by removing all zeroes from $\alpha$. 

 We consider three involutions on compositions. The \emph{complement} of a composition $\alpha$ is defined $\alpha^c = comp(set(\alpha)^c)$. Intuitively, if the composition $\alpha$ is represented as blocks of stars separated by bars, then $\alpha^c$ is the composition given by 
 placing bars in exactly the places where $\alpha$ does not have bars. The \emph{reverse} of $(\alpha_1, \ldots, \alpha_k)$, denoted $\alpha^r$, is defined as $(\alpha_k, \ldots, \alpha_1)$. The \emph{transpose} of $\alpha$ is defined $\alpha^t = (\alpha^r)^c = (\alpha^c)^r$. Additionally, the \emph{conjugate} of a partition $\lambda$, denoted $\lambda'$, is defined by flipping the diagram of $\lambda$ over the diagonal. The conjugate of a partition is also sometimes called the transpose, but note that $\lambda^t$ is not always equal to $\lambda'$ \cite{Luoto2013}. 

 \begin{ex}
     Consider the partition $\beta = (3,2)$. Then, $$\beta^c = (1,1,2,1), \quad \beta^r = (2,3), \quad \beta^t = (1,2,1,1), \text{\quad and \quad} \beta' = (2,2,1).$$ Note that the first three maps are involutions on all compositions, whereas the final map is an involution only on partitions.
 \end{ex}

Let $\alpha = (\alpha_1, \ldots, \alpha_k)$ and $\beta = (\beta_1, \ldots, \beta_j)$ be two compositions. The \emph{concatenation} of $\alpha$ and $\beta$ is given by $\alpha \cdot \beta = (\alpha_1, \ldots, \alpha_k, \beta_1, \cdots, \beta_j)$ and the \emph{near-concatenation} is given by $\alpha \odot \beta = (\alpha_1, \ldots, \alpha_{k-1}, \alpha_k + \beta_1, \beta_2, \ldots, \beta_j)$. Define $sort(\alpha)$ as the unique partition that can be obtained by rearranging the parts of $\alpha$ into weakly decreasing order.  Under the \textit{refinement order} $\preceq$ on compositions of size $n$, we say $\alpha \preceq \beta$ if and only if $\{\beta_1, \beta_1+\beta_2, \ldots, \beta_1 + \cdots + \beta_j = n\} \subseteq \{\alpha_1, \alpha_1 + \alpha_2, \ldots, \alpha_1 + \cdots + \alpha_k = n\}$. Under \emph{dominance order} $\subseteq$ on compositions, we say $\alpha \subseteq \beta$ if and only if $k \leq j$ and $\alpha_i \leq \beta_i$ for $1 \leq i \leq k$. 

 \begin{ex} Let $\alpha = (1,2,3,1)$ and $\beta = (3,2)$. Then, $$\alpha \cdot \beta = (1,2,3,1,3,2) \qquad \qquad \alpha \odot \beta = (1,2,3,4,2) \qquad \qquad sort(\alpha) = (3,2,1,1),$$
 $$(1,1,1,1) \preceq (1,2,1) \preceq (1,3) \preceq (4) \qquad \qquad (1,1,1) \subseteq (2,1,1,1) \subseteq (2,3,1,2).$$
 \end{ex}

There is a natural bijection between ordered subsets of $[n-1] = \{1,2, \ldots, n-1 \}$ and compositions of $n$. For an ordered set $S = \{s_1,\ldots,s_k\} \subseteq \{[n-1]\}$, let $comp(S) = (s_1, s_2-s_1, \ldots, s_k-s_{k-1}, n-s_k)$ and for a composition $\alpha = (\alpha_1,\ldots,\alpha_j)$, let $set(\alpha) = \{\alpha_1, \alpha_1+\alpha_2, \ldots, \alpha_1 + \alpha_2 + \ldots + \alpha_{j-1}\}$. A \textit{permutation} $\sigma$ of a set is a bijection from the set to itself. The permutation $\sigma$ of $[n]$ is written in one-line notation as $\sigma(1)\sigma(2) \cdots \sigma(n)$ \cite{EC2}. 

\subsection{Symmetric functions and the Schur basis}

The \emph{symmetric functions} are formal power series such that $f(x_1, x_2, \ldots ) = f(x_{\sigma(1)}, x_{\sigma(2)}, \ldots)$ for any permutation $\sigma$ of $\mathbb{N}$. They form a ring and specifically a graded Hopf algebra,  which we take over $\mathbb{Q}$ and denote $Sym$. For a partition $\lambda \vdash n$, the \emph{monomial symmetric function} is defined as $$m_{\lambda} = \sum_{sort(\alpha) = \lambda} x^{\alpha},$$ where the sum runs over all unique compositions $\alpha$ that are permutations of the entries of $\lambda$. For $\lambda = (\lambda_1, \ldots, \lambda_k) \vdash n$, the \emph{elementary symmetric function} $e_{\lambda}$ is defined as $$e_n = m_{1^n}= \sum_{i_1 < \cdots < i_n} x_{i_1} \cdots x_{i_n}, \text{\quad where \quad} e_{\lambda} = e_{\lambda_1}e_{\lambda_2} \cdots e_{\lambda_k}.$$ The \emph{complete homogeneous symmetric function} is defined as
$$h_n = \sum_{\lambda \vdash n} m_{\lambda} = \sum_{i_1 \leq \cdots \leq i_n}x_{i_1} \cdots x_{i_n} \text{\quad where \quad} h_{\lambda} = h_{\lambda_1} \cdots h_{\lambda_k}.$$ 

The sets of functions above form bases of $Sym$. The most important basis for our purposes, however, is the Schur basis. We first define the Schur functions combinatorially. For a partition $\lambda \vdash n$, a \emph{semistandard Young tableau} (SSYT) of shape $\lambda$ is a filling of the diagram of $\lambda$ with positive integers such that the rows are weakly increasing from left to right and the columns are strictly increasing from top to bottom. The \emph{size} of a SSYT is its number of boxes, $n = \sum_i \lambda_i$, and its \emph{type} is the weak composition $type(T)=\beta = (\beta_1,\ldots,\beta_j)$ where $T$ has $\beta_i$ boxes containing an $i$ for all positive integers $i$. A \emph{standard Young tableau} (SYT) of size $n$ is a Young tableau in which the numbers $1$ through $n$ each appear exactly once.  A SSYT $T$ of type $\beta = (\beta_1, \ldots, \beta_k)$ is associated with the monomial $x^T = x_1^{\beta_1} \cdots x_k^{\beta_k}$.

\begin{defn}
For a partition $\lambda$, the \emph{Schur symmetric function} is defined as $$s_{\lambda} = \sum_T x^T,$$ where the sum runs over all semistandard Young tableaux $T$ of shape $\lambda$ with entries in $\mathbb{Z}_{> 0}$. These functions also form a basis of $Sym$. 
\end{defn}

A \emph{skew semistandard Young tableau} of shape $\lambda/\mu$ is a skew shape $\lambda/\mu$ filled positive integers such that the rows are weakly increasing left to right and the columns are strictly increasing top to bottom.  Note that the (removed) inner boxes of shape $\mu$ are not filled with integers. The \emph{skew Schur function} is defined as $$s_{\lambda/\mu} = \sum_T x^T,$$ where the sum runs over all skew SSYT of shape $\lambda/\mu$. The Schur and skew Schur functions have multiple equivalent definitions. We first review the definition using creation operators in $Sym$ and then the definition using a Jacobi-Trudi rule. 

$Sym$ is a self-dual algebra with the bilinear inner product $\langle \cdot, \cdot \rangle: Sym \rightarrow Sym$ defined by $\langle m_{\lambda}, h_{\mu} \rangle = \delta_{\lambda, \mu}$. For each element $f \in Sym$, there is an operator, known as the \emph{perp operator}, $f^{\perp}: Sym \times Sym \rightarrow \mathbb{Q}$ defined according to the relation $\langle fg,h \rangle = \langle g, f^{\perp}h \rangle $ for all $g,h \in Sym$.  Thus, given $\{a_{\lambda}\}_{\lambda}$ and $\{b_{\lambda}\}_{\lambda}$ are any two dual bases of $Sym$, we have $$f^{\perp}(g) = \sum_{\lambda} \langle g, f a_{\lambda} \rangle b_{\lambda}.$$ This expression expands the output of a perp operator on a function in terms of any basis with coefficients given by values of the inner product on $Sym$.

\begin{defn} For an integer $m$, the \emph{Bernstein creation operator} $B_m: Sym_n \rightarrow Sym_{m+n}$ is defined by $$B_m = \sum_{i \geq 0} (-1)^i h_{m+i}e^{\perp}_{i} = \sum_{i \geq 0} (-1)^i h_{m+i} m^{\perp}_{1^i}.$$
\end{defn}

As shown below, applying the Bernstein creation operators in sequence produces a Schur symmetric function. Note that this definition also yields Schur functions defined on compositions that are not partitions, although these functions are not basis elements.

\begin{thm} \cite{Zel} For all tuples $\alpha \in \mathbb{Z}^m$, $$s_{\alpha} = B_{\alpha_1} B_{\alpha_2} \cdots B_{\alpha_m}(1),$$ where $s_{0} = 1$ and $s_{-i}=0$ for $i > 0$.     
\end{thm}

The skew Schur functions (of which the Schur functions are a subset) can be equivalently defined with a \emph{Jacobi-Trudi Rule.} This rule expresses the skew Schur functions as matrix determinants in terms of the complete homogeneous basis.

\begin{thm}\cite{EC2} Let $\lambda = (\lambda_1, \ldots, \lambda_k)$ and $\mu=(\mu_1, \ldots, \mu_k) \subseteq \lambda$. Then, $$s_{\lambda/\mu} = det(h_{\lambda_i - \mu_j - i + j})_{1 \leq i,j \leq k},$$ where $h_{0}=1$ and $h_{-i} = 0$ for any $i > 0$. 
\end{thm}

While the Schur functions have many other properties, there are two that are most important for our purposes. First, $Sym$ possesses an involutive endomorphism $\omega: Sym \rightarrow Sym$ defined by $\omega(e_{\lambda}) = h_{\lambda}$. Applied to a Schur function it yields $\omega(s_{\lambda}) = s_{\lambda'}$. Second, the product of two Schur functions expanded in the Schur basis is given by $$s_{\lambda}s_{\mu} = \sum_{\nu \vdash |\lambda| + |\mu|} c_{\lambda, \mu}^{\nu} s_{\nu},$$ where $c^{\nu}_{\lambda, \mu}$ are the \emph{Littlewood-Richardson coefficients.} In general, we refer to this type of coefficients as \emph{structure coefficients}. These coefficients are central in the representation theory applications of the Schur basis among others. They also appear in the skew Schur functions, alongside the $s^{\perp}$ operator. That is $$s^{\perp}_{\mu} (s_{\lambda}) = s_{\lambda/\mu} = \sum_{\nu \models |\lambda|-|\mu|} c^{\lambda}_{\mu, \nu} s_{\nu}.$$ For more details on the symmetric functions and the Schur functions, see \cite{berg18, grinberg, EC2}.

\subsection{Quasisymmetric and noncommutative symmetric functions}
A \emph{quasisymmetric function} $f(x)$ with rational coefficients is a formal power series of the form 
$f(x) = \sum_{\alpha} b_{\alpha}x^{\alpha},$ where $\alpha$ is a weak composition of a positive integer,
 $x^{\alpha} = x_1^{\alpha_1}\ldots x_k^{\alpha_k}$, $b_{\alpha} \in \mathbb{Q}$, and the coefficients of monomials $x_{i_1}^{a_1}\ldots x_{i_k}^{a_k}$ and $x_{j_1}^{a_1}\ldots x_{j_k}^{a_k}$ are equal if $i_1 < \ldots  < i_k$ and $j_1 < \ldots  < j_k$. The algebra of quasisymmetric functions, denoted $QSym$, admits a Hopf algebra structure.  The two most common bases (indexed by compositions) of $Qsym$ are as follows. For a composition $\alpha$, the \textit{monomial quasisymmetric function} $M_{\alpha}$ is defined as $$M_{\alpha} = \sum_{i_1<\ldots <i_k}x_{i_1}^{\alpha_1}\ldots x_{i_k}^{\alpha_k},$$ where the sum runs over strictly increasing sequences of $k$ positive integers $i_1, \ldots, i_k \in \mathbb{Z}_{>0}$.  
The \textit{fundamental quasisymmetric function} $F_{\alpha}$ is defined as $$F_{\alpha} = \sum_{\beta \preceq \alpha}M_{\beta},$$  
where the sum runs over compositions $\beta$ that refine $\alpha$. The fundamental functions are also denoted $L_{\alpha}$ in the literature \cite{EC2}. The quasisymmetric functions were formally introduced by Gessel in his work \cite{gessel} on multipartite $P$-partitions. They are the terminal object in the category of combinatorial Hopf algebras, connect to the representation theory of $0$-Hecke algebras, and support new methods for solving problems of Schur-positivity \cite{Agu, mason}.

The algebra of \emph{noncommutative symmetric functions}, written \textit{NSym}, is the Hopf algebra dual to $Qsym$. It was introduced in \cite{gelfand} by Gelfand, Krob, Lascoux, Leclerc, Retakh, and Thibon. \textit{NSym} can be defined as the free algebra with generators $\{H_1,H_2, \ldots \}$ and no relations, meaning the generators do not commute,
$$\textit{NSym} = \mathbb{Q}  \left\langle H_1, H_2, \ldots \right\rangle.$$ 

For a composition $\alpha = (\alpha_1, \ldots, \alpha_k)$,  define the \emph{complete homogeneous noncommutative symmetric function} by $$H_{\alpha} = H_{\alpha_1}H_{\alpha_2} \ldots H_{\alpha_k}.$$ Then, the set $\{H_{\alpha}\}_{\alpha}$ forms a basis of \textit{NSym}. \textit{NSym} and $QSym$ are dually paired by the bilinear inner product defined by $\langle H_{\alpha}, M_{\beta} \rangle = \delta_{\alpha, \beta}$ for all compositions $\alpha, \beta$. The dual relationship between bases of \textit{NSym} and $QSym$ manifests as follows.

\begin{prop}\label{hopf_dual} \cite{Hoffman} Let $\mathcal{A}$ and $\mathcal{B}$ be dually paired algebras and let $\{a_i\}$ be a basis of $\mathcal{A}$. A basis $\{b_i\}$ of $\mathcal{B}$ is the unique basis that is dual to $\{a_i\}$ if and only if the following relationship holds for any pair of dual bases $\{c_i\}$ in $\mathcal{A}$ and $\{d_i\}$ in $\mathcal{B}$: $$a_i = \sum_{j} k_{i,j} c_j \text{\quad and \quad} d_j = \sum_{i} k_{i,j} b_i.$$
\end{prop}

We consider two other classical bases of \textit{NSym} from \cite{gelfand}. For a composition $\alpha$, the \emph{ribbon noncommutative symmetric function} is defined as $$R_{\alpha}=\sum_{\alpha \preceq \beta }(-1)^{\ell{(a)}-\ell{(\beta)}}H_{\beta}, \text{\quad and so \quad} H_{\beta} = \sum_{\alpha \preceq \beta} R_{\alpha}.$$ The ribbon functions are a basis of \textit{NSym} dual to the fundamental basis of $QSym$, meaning $\langle R_{\alpha}, F_{\beta} \rangle = \delta_{\alpha, \beta}$.  For a composition $\alpha$, the \emph{elementary noncommutative symmetric function} is defined as $$E_{\alpha} = \sum_{\beta \preceq \alpha} (-1)^{|\alpha|-\ell(\beta)}H_{\beta}, \text{ \quad and so \quad } H_{\beta} = \sum_{\alpha \preceq \beta} (-1)^{|\beta|- \ell(\alpha)} E_{\alpha}.$$ 
 
Multiplication in \textit{NSym} on these bases is given by 
\begin{equation}\label{HR_mult}
    H_{\alpha} H_{\beta} = H_{\alpha \cdot \beta} \text{,\qquad} R_{\alpha}R_{\beta} = R_{\alpha \cdot \beta} + R_{\alpha \odot \beta} \text{,\qquad and \qquad} E_{\alpha} E_{\beta} = E_{\alpha \cdot \beta}.
\end{equation}

The \emph{forgetful map} $\chi: \textit{NSym} \rightarrow Sym$ maps $$\chi(H_{\alpha}) = h_{sort(\alpha)},$$ extended linearly to map all elements in \textit{NSym} to their commutative image in $Sym$. The forgetful map is a Hopf algebra morphism that is dual to the inclusion of $Sym$ to $QSym$ \cite{BRRZ}. 

A \emph{Schur-like basis} $\{\mathcal{S}_{\alpha}\}_{\alpha}$ of \textit{NSym} is a basis where $\chi(\mathcal{S}_{\lambda}) = s_{\lambda}$ for any partition $\lambda$. One might also consider a basis of \textit{NSym} to be Schur-like if the commutative image of some subset is the Schur functions or if the basis has sufficiently Schur-like properties. A Schur-like basis of $QSym$ is generally one that is dual to a Schur-like basis of \textit{NSym} and is usually defined combinatorially in terms of tableaux that resemble or generalize the semistandard Young tableaux. The interested reader should see \cite{berg18} and \cite{Luoto2013} for information on other Schur-like bases: the immaculate and Young noncommutative Schur bases and their duals, respectively. For more details on the quasisymmetric and noncommutative symmetric functions see \cite{gelfand, grinberg, mason}.

\subsection{The extended Schur functions}

The extended Schur functions of $QSym$ were introduced by Assaf and Searles in \cite{Assa} as the stable limits of polynomials related to Kohnert diagrams, and defined independently by Campbell, Feldman, Light, Shuldiner, and Xu in  \cite{Camp} as the duals to the shin functions. We use the name ``extended Schur functions'' but otherwise retain the notation and terminology of the dual shin functions.

\begin{defn} Let $\alpha$ and $\beta$ be a composition and a weak composition of $n$, respectively. A \emph{shin-tableau} of shape $\alpha$ and type $\beta$ is a labeling of the boxes of the diagram of $\alpha$ by positive integers such that the number of boxes labeled by $i$ is $\beta_i$, the sequence of entries in each row is weakly increasing from left to right, and the sequence of entries in each column is strictly increasing from top to bottom. 
\end{defn}

Note that shin-tableaux are a direct generalization of semistandard Young tableaux to composition shapes.

\begin{ex} The shin-tableaux of shape $(3,4)$ and type $(1,2,1,1,2)$ are
$$\scalebox{.85}{\begin{ytableau}
    1 & 2 & 2\\
    3 & 4 & 5 & 5\\
\end{ytableau}
\qquad
\begin{ytableau}
    1 & 2  & 3\\
    2 & 4 & 5 & 5\\
\end{ytableau}
\qquad
\begin{ytableau}
    1 & 2  & 4\\
    2 & 3 & 5 & 5  
\end{ytableau}}
$$
\end{ex}

A shin-tableau \let\thefootnote\relax\footnote{$\shin$ is the Hebrew character \emph{shin}.}of shape $\alpha \models n$ is \emph{standard} if each number $1$ through $n$ appears exactly once. The \emph{descent set} is defined as $Des_{\shin}(U) = \{i : i+1 \text{ is strictly below $i$ in $U$} \}$ for a standard shin-tableau $U$. Each entry $i$ in $Des_{\shin}(U)$ is called a \emph{descent} of $U$. The \emph{descent composition} of $U$ is defined $co_{\shin}(U) = (i_1, i_2 - i_1, \ldots, i_{d}-i_{d-1}, n-i_{d})$ for $Des_{\shin}(U) = \{i_1, \ldots, i_d\}$. Equivalently, the descent composition is found by counting the number of entries in $U$ (in the order they are numbered) after each descent to the next one, including that descent.  The \emph{shin reading word} of a shin-tableau $T$, denoted $rw_{\shin}(T)$, is the word obtained by reading the rows of $T$ from left to right starting with the bottom row and moving up. To \emph{standardize} a shin-tableau $T$, replace the $1$'s in $T$ with $1,2,\ldots$ in the order they appear in $rw_{\shin}(T)$, then the $2$'s starting with the next consecutive number, etc.

A shin-tableau $T$ of type $\beta = (\beta_1, \ldots, \beta_{\ell})$ is associated with the monomial $x^T = x_1^{\beta_1}x_{2}^{\beta_2} \cdots x_{\ell}^{\beta_{\ell}}$.  This yields the following polynomials.

\begin{defn}
For a composition $\alpha$, the \emph{extended Schur function} is defined as $$\shin^*_{\alpha} = \sum_T x^T,$$ where the sum runs over shin-tableaux $T$ of shape $\alpha$. 
\end{defn}

 The extended Schur functions have positive expansions into the monomial and fundamental bases in terms of shin-tableaux \cite{Assa, Camp}. For a composition $\alpha$, 
\begin{equation}
    \shin^*_{\alpha} = \sum_{\beta} \mathcal{K}_{\alpha, \beta}M_{\beta} \text{\qquad and \qquad} \shin^*_{\alpha} = \sum_{\beta } \mathcal{L}_{\alpha,\beta} F_{\beta },
\end{equation} where $\mathcal{K}_{\alpha, \beta}$ denotes the number of shin-tableaux of shape $\alpha$ and type $\beta$, and $\mathcal{L}_{\alpha, \beta}$ denotes the number of standard shin-tableaux of shape $\alpha$ with descent composition $\beta$. Note that Assaf and Searles showed in \cite{Assa} that $\shin^*_{\alpha} = F_{\alpha}$ if and only if $\alpha$ is a reverse hook shape $(1^k, m)$ for $k, m \geq 0$.

\begin{ex} The $F$-expansion of $\shin^*_{(2,3)}$ and its relevant standard shin-tableaux are: 
$$\shin^*_{(2,3)} = F_{(2,3)} + F_{(1,2,2)} \qquad \scalebox{.75}{\begin{ytableau}
        1 & 2  \\
        3 & 4 & 5
    \end{ytableau}
    \quad \quad 
    \begin{ytableau}
        1 & 3 \\
        2 & 4 & 5
    \end{ytableau}}$$
\end{ex}

From these definitions, it follows that $\shin^*_{\lambda}= s_{\lambda}$ for a partition $\lambda$. One consequence of this fact is that the product of any two elements of the extended Schur basis indexed by partitions expands positively into the extended Schur basis with the Littlewood-Richardson coefficients. That is, for partitions $\lambda$ and $\mu$,

$$\shin^*_{\lambda} \shin^*_{\mu} = \sum_{\nu \vdash |\lambda| + |\mu|} c^{\nu}_{\lambda, \mu} \shin^*_{\nu},$$ where $c_{\lambda, \mu}^{\nu}$ are the Littlewod-Richardson coefficients.

\subsection{The Shin functions}

The shin basis of \textit{NSym} was introduced in \cite{Camp} by Campbell, Feldman, Light, Shuldiner, and Xu. Let $\alpha$ and $\beta$ be compositions.  Then $\beta$ is said to differ from $\alpha$ by a \emph{shin-horizontal strip} of size $r$, denoted $\alpha \subset^{\shin}_r \beta$,  provided for all $i$, we have $\beta_i \geq \alpha_i$, $|\beta| = |\alpha|+r$, and for all indices $i \in \mathbb{N}$ if $\beta_i > \alpha_i$ then for all $j>i$, we have $\beta_j \leq \alpha_i$. The shin functions are defined recursively with a right Pieri rule using shin-horizontal strips. 

 \begin{defn}\label{shin_r_pieri}
     The shin basis $\{\shin_{\alpha}\}_{\alpha}$ is defined as the unique set of functions such that  \begin{equation}
         \shin_{\alpha}H_r = \sum_{\alpha \subset^{\shin}_r \beta}\shin_{\beta},
     \end{equation} where the sum runs over all compositions $\beta$ which differ from $\alpha$ by a shin-horizontal strip of size $r$.
 \end{defn}

Intuitively, the $\beta$'s in the right-hand side are given by taking the diagram of $\alpha$ and adding $r$ blocks on the right such that if you add boxes to some row $i$ then no row below $i$ is longer than the original row $i$. This is referred to as the \emph{overhang rule}.

\begin{ex} The following example can be pictured with the diagrams below.
$$\shin_{(2,3,1)}H_{(2)} = \shin_{(2,3,1,2)} + \shin_{(2,3,2,1)} + \shin_{(2,3,3)} +\shin_{(2,4,1,1)} + \shin_{(2,4,2)} + \shin_{(2,5,1)}$$\vspace{-2mm}

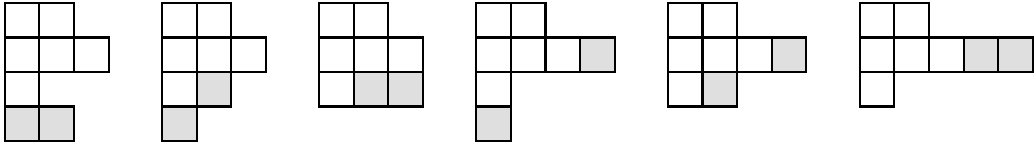
    \center{
\scalebox{.85}{
    \begin{ytableau}
        *(white) & *(white) \\
        *(white) & *(white) & *(white) \\
        *(white) \\
        *(lightgray!50) & *(lightgray!50)
    \end{ytableau}
    \quad \quad
     \begin{ytableau}
        *(white) & *(white) \\
        *(white) & *(white) & *(white) \\
        *(white) & *(lightgray!50)\\
        *(lightgray!50)
    \end{ytableau}
    \quad \quad
     \begin{ytableau}
        *(white) & *(white) \\
        *(white) & *(white) & *(white) \\
        *(white) & *(lightgray!50) & *(lightgray!50)\\
    \end{ytableau}
    \quad \quad
     \begin{ytableau}
        *(white) & *(white) \\
        *(white) & *(white) & *(white) & *(lightgray!50)\\
        *(white) \\
        *(lightgray!50)
    \end{ytableau}
    \quad \quad
     \begin{ytableau}
        *(white) & *(white) \\
        *(white) & *(white) & *(white) & *(lightgray!50)\\
        *(white) & *(lightgray!50)\\
    \end{ytableau}
    \quad \quad
     \begin{ytableau}
        *(white) & *(white) \\
        *(white) & *(white) & *(white) & *(lightgray!50) & *(lightgray!50)\\
        *(white) \\
    \end{ytableau}}}
\end{ex}\vspace{2mm}

\begin{rem} One computes $\shin_{(3,2)}$ recursively in terms of the complete homogeneous basis starting with $\shin_{\emptyset}H_{(n)} = \shin_{(n)}$ so $H_{(n)}= \shin_{(n)}$ for all $n$. Then, $\shin_{(4)}H_{(1)} = \shin_{(5)}+ \shin_{(4,1)}$ so $\shin_{(4,1)} = H_{(4)}H_{(1)} - H_{(5)}$.  Next, $\shin_{(3)}H_{(2)} = \shin_{(5)} + \shin_{(4,1)} + \shin_{(3,2)}$ so $\shin_{(3,2)} = H_{(3)}H_{(2)} - H_{(5)}- H_{(4)}H_{(1)} + H_{(5)} = H_{(3,2)} - H_{(4,1)}.$  
\end{rem}

Repeated application of this right Pieri rule yields the expansion of a complete homogeneous noncommutative symmetric function in terms of the shin functions. This expansion verifies that the extended Schur functions and the shin functions are dual bases by Proposition \ref{hopf_dual}, and allows for the expansion of the ribbon functions into the shin basis dually to the expansion of the extended Schur functions into the fundamental basis.  For any composition $\beta$,  \begin{equation}\label{HR_to_shin}
H_{\beta} = \sum_{\alpha} \mathcal{K}_{\alpha, \beta} \shin_{\alpha} \text{\qquad and \qquad} R_{\beta} = \sum_{\alpha} \mathcal{L}_{\alpha,\beta} \shin_{\alpha}.
\end{equation}

The shin basis also has a nice combinatorial rule for multiplication by a ribbon function using skew tableaux.

\begin{defn}
    For compositions $\alpha =(\alpha_1, \ldots, \alpha_k)$ and $\beta = (\beta_1, \ldots, \beta_{\ell})$ such that $\beta \subseteq \alpha$, a filling of the skew shape $\alpha/\beta$ with positive integers is a
    \emph{skew shin-tableau} if:
    \begin{enumerate}
        \item[(i)] each row is weakly increasing from left to right,
        \item[(ii)] each column is strictly increasing from top to bottom, and
        \item[(iii)] if $\alpha_i> \beta_i$ for any  $1 \leq i \leq k$, then $\beta_j < \beta_i$ for all $j > i$.
    \end{enumerate} A skew shin-tableau is \emph{standard} if it contains the numbers $1, \ldots, |\alpha|-|\beta|$ each exactly once. Note that the boxes associated with $\beta$ are not filled.
\end{defn}

\begin{ex}
    The three leftmost diagrams are skew shin-tableaux while the rightmost diagram is not because there are boxes in the first row sitting directly above skewed-out boxes in the second row, which violates condition (iii).

    $$ \text{Skew shin:\quad}
    \scalebox{.85}
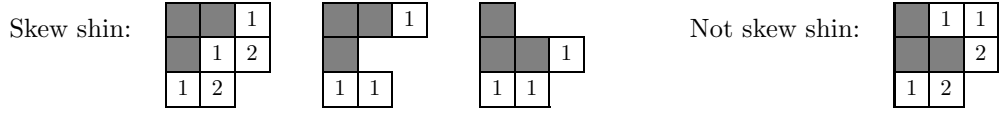
{
    \begin{ytableau}
        *(gray) & *(gray) & 1\\
        *(gray) &  1 & 2\\
        1 & 2
    \end{ytableau}
    \qquad 
    \begin{ytableau}
         *(gray) & *(gray) & 1\\
        *(gray)   \\
        1 & 1
    \end{ytableau}
    \qquad
    \begin{ytableau}
        *(gray)\\ 
        *(gray) & *(gray) & 1 \\
        1 & 1
    \end{ytableau}}
     \qquad
     \qquad
     \text{Not skew shin:\quad }
     \scalebox{.85}{
     \begin{ytableau}
         *(gray) & 1 & 1 \\
        *(gray) & *(gray) & 2 \\
        1 & 2
     \end{ytableau}}
    $$

\end{ex}

The notions of type, descents, and descent composition for skew shin-tableaux are the same as those for shin-tableaux. There are combinatorial rules for the multiplication of the shin basis by ribbon functions in terms of these skew shin-tableaux.

\begin{thm} \label{shin_right_R_mult} \cite{Camp} For all compositions $\alpha$ and $\beta$, $$\shin_{\alpha}R_{\beta} = \sum_{\gamma \models |\alpha|+|\beta|}\sum_{\mathcal{T}}\shin_{\gamma},$$ where the inner sum is over all standard skew shin-tableaux $\mathcal{T}$ of skew shape $\gamma/\alpha$ such that $co_{\shin}(\mathcal{T})=\beta$.
\end{thm}

Recall that the extended Schur functions have the special property that  $\shin^*_{\lambda} = s_{\lambda}$ for partitions $\lambda$. Since the forgetful map $\chi$ is dual to the inclusion map from $Sym$ to $QSym$, we have \begin{equation}
    \chi(\shin_{\alpha}) = \sum_{\lambda} \text{(coefficient of $\shin^*_{\alpha}$ in $s_{\lambda}$)} s_{\lambda} = \begin{cases}
        s_{\lambda} & \text{ if $\alpha = \lambda$,}\\
        0 & \text{otherwise.}
    \end{cases} 
\end{equation}
It follows that $\chi(\shin_{\lambda}) = s_{\lambda}$ when $\lambda$ is a partition and $\chi(\shin_{\alpha}) = 0$ otherwise. Another interesting feature of the shin functions is their relationship with the other two primary Schur-like bases: the immaculate functions and the Young noncommutative Schur functions. For any partition $\lambda$, the immaculate function $\mathfrak{S}_{\lambda}$ always equals the Young noncommutative Schur function $\hat{\mathbf{s}}_{\lambda}$, however the shin function $\shin_{\lambda}$ often differs. The first examples of this appear for some partitions $\lambda$ with $|\lambda|=7$, for example $\shin_{(2,2,2,1)}$. Additionally, while the dual immaculate functions expand positively into the Young quasisymmetric Schur functions, the extended Schur basis is not guaranteed to expand positively in either, or vice versa. Positive expansions of certain extended Schur functions into the Young quasisymmetric Schur basis are given by Marcum and Niese in \cite{MarNie}.

\section{A Jacobi-Trudi rule for certain shin functions}\label{section_creation}

The Schur functions and the immaculate functions can both be defined in terms of creation operators. In fact, the immaculate basis was originally defined in terms of the \emph{noncommutative Bernstein operators} \cite{berg18}. It is using these operators that one can prove various properties of the immaculate basis including the Jacobi-Trudi rule \cite{berg18}, a left Pieri rule \cite{bergerondualpieri}, a combinatorial interpretation of the inverse Kostka matrix \cite{AllMas}, and a partial Littlewood Richardson rule \cite{berg_structures}. Here we give similar creation operators for certain shin functions which then allow us to define a \emph{Jacobi-Trudi rule}. This rule is especially useful because there is currently no other combinatorial way to expand shin functions into the complete homogeneous basis.

\begin{defn} 
For a composition $\alpha = (\alpha_1, \alpha_2, \ldots, \alpha_k)$ with $k \geq 1$ and a positive integer $m$, define the linear operator $\beth_m$ on the complete homogeneous basis by \let\thefootnote\relax\footnote{$\beth$ is the hebrew character \emph{beth}.}
$$\beth_m(1) = H_m \text{\qquad and \qquad} \beth_m(H_{\alpha}) = H_{(m, \alpha_1, \alpha_2, \ldots)} - H_{(\alpha_1, m, \alpha_2, \ldots)}$$
\end{defn}

\begin{ex}
For example, $\beth_{2}(H_{(3,1)}) = H_{(2,3,1)} - H_{(3,2,1)}$. 
\end{ex}

Using this definition, we can describe how these operators act on certain shin functions.

\begin{thm}
If $\alpha = (\alpha_1, \alpha_2, \ldots, \alpha_k)$ with $k\geq 1$ and $0< m < \alpha_1$, then
    $\beth_m (\shin_{\alpha}) = \shin_{(m, \alpha)}.$ 
\end{thm}

\begin{proof} 

Define $f_{\alpha} = \beth_{\alpha_1}(\shin_{\bar{\alpha}})$ for any nonempty $\alpha = (\alpha_1, \ldots, \alpha_k)$ where $\alpha_1 < \alpha_2$ and $\bar{\alpha} = (\alpha_2, \ldots, \alpha_k)$.  We want to show that $f_{\alpha} = \shin_{\alpha}$ for all such $\alpha$. In the case that $\ell(\alpha)=1$, we simply have $f_{\alpha} = \beth_{\alpha_1}(1) = H_{\alpha_1} = \shin_{\alpha_1}$. 

Now, consider only $\alpha$ where $\ell(\alpha)>1$.  First, we show that the $f_{\alpha}$'s satisfy the right Pieri rule for shin functions. 
Observe that, for a positive integer $m$ and a nonempty composition $\gamma = (\gamma_1, \gamma_2, \ldots )$, 
$$\beth_m (H_{\gamma})H_s = H_{m, \gamma_1, \gamma_2, \ldots, s} - H_{\gamma_1, m, \gamma_2, \ldots, s} = \beth_m(H_{\gamma \cdot s}).$$  Since $\beth_m$ is a linear operator, $\beth_m(\shin_{\gamma})$ is equivalent to the sum of $\beth_m$ applied to each term in the H-expansion of $\shin_{\gamma}$. Thus,  $\beth_m(\shin_{\gamma})H_s = \beth_{m}(\shin_{\gamma}H_s)$ for any composition $\gamma$ and positive integer $s$. 
For $\alpha = (\alpha_1, \ldots, \alpha_k)$ where $\alpha_1 < \alpha_2$ and $\ell(\alpha)>1$, it follows that $f_{\alpha}H_m = \beth_{\alpha_1}(\shin_{\bar{\alpha}})H_m = \beth_{\alpha_1}(\shin_{\bar{\alpha}}H_m)$. We expand this expression further using the shin right Pieri rule (Definition \ref{shin_r_pieri}). Notice that if $\alpha_1 < \alpha_2$ and $\bar{\alpha} \subset^{\shin}_{m} \beta$ for a composition $\beta = (\beta_1, \ldots, \beta_{k-1})$, then 
$\alpha_2 \leq \beta_1$ which implies $\alpha_1 < \beta_1$. Thus, 
$$f_{\alpha}H_m = \beth_{\alpha_1} \left(\sum_{\bar{\alpha} \subset^{\shin}_{m} \beta} \shin_{\beta}\right) = \sum_{\bar{\alpha} \subset^{\shin}_{m} \beta} \beth_{\alpha_1}(\shin_{\beta}) = \sum_{\bar{\alpha} \subset^{\shin}_{m} \beta} f_{\alpha_1 \cdot \beta}.$$ 

Observe that every $\gamma$ where $\alpha \subset^{\shin}_{m} \gamma$ has $\alpha_1\leq \gamma_1$, and we can only have $\alpha_1 < \gamma_1$ if $\alpha_j \leq \gamma_j < \alpha_1$ for all $j>1$.  For $\alpha$, we know $\alpha_2 > \alpha_1$, implying that $\gamma_2$ could never be less than $\alpha_1$.  Thus, $\gamma_1 > \alpha_1$ is not an option because it would violate the overhang rule, meaning has $\gamma_1 = \alpha_1$. It follows that the set of $\gamma$ such that $\alpha \subset^{\shin}_{m} \gamma$ is the same as the set of $\alpha_1 \cdot \beta$ such that $\bar{\alpha} \subset^{\shin}_{m} \beta$. Therefore,
\begin{equation}\label{beth_pieri}
    f_{\alpha}H_m = \sum_{\alpha \subset^{\shin}_m \gamma} f_{\gamma}.
\end{equation}

Next, we show by recursive calculation that $f_{\alpha} = \shin_{\alpha}$. Let $\alpha = (\alpha_1, \ldots, \alpha_k)$ where $\ell(\alpha)>2$ and $\alpha_1 < \alpha _2$.  Let $\underline{\alpha} = (\alpha_1, \ldots, \alpha_{k-1})$ and $\underline{\bar{\alpha}} = (\alpha_2, \ldots, \alpha_{k-1})$, and note that because we assumed $\alpha_1 < \alpha_2$, we have $\beth_{\alpha_1}(\shin_{\underline{\bar{\alpha}}}) = f_{\underline{\alpha}}$.  Observe that $f_{\underline{\alpha}}H_{\alpha_k} = \sum_{\underline{\alpha} \subset^{\shin}_{\alpha_k} \beta} f_{\beta}$ by Equation \eqref{beth_pieri}, and rearranging this expression yields
\begin{equation}\label{recursive_shin}
f_{\alpha} = f_{\underline{\alpha}}H_{\alpha_k} - \sum_{\substack{\underline{\alpha} \subset^{\shin}_{\alpha_k} \beta \\
\beta \not= \alpha}} f_{\beta}.
\end{equation}
Thus, we have a recursive formula for $f_{\alpha}$ in terms of $f_{\underline{\alpha}}$ and $f_{\beta}$ where either $\ell(\beta)<\ell(\alpha)$ or $\ell(\beta) = \ell(\alpha)$ and $\beta_k < \alpha_k$. That is to say, our recursive formula is defined in terms of $f$ indexed either by a composition shorter than $\alpha$ or the same length as $\alpha$ but with a shorter final row. 

It simply remains to show that $f_{\alpha} = \shin_{\alpha}$ when $\ell(\alpha) = 2$ as our base case and this recursive definition implies $f_{\alpha} = \shin_{\alpha}$ because it matches the recursive definition for $\shin_{\alpha}$ given by manipulating the right Pieri rule of Definition \ref{shin_r_pieri} in the same way (and has the same base case). Let $\alpha = (\alpha_1, \alpha_2)$ with $\alpha_1 < \alpha_2$.  Then we have $$f_{\alpha} = \beth_{\alpha_1}(\shin_{\alpha_2}) = \beth_{\alpha_1}(H_{\alpha_2}) = H_{\alpha_1}H_{\alpha_2} - H_{\alpha_2}H_{\alpha_1}.$$ Next, observe by Definition \ref{shin_r_pieri} we have $$H_{\alpha_1}H_{\alpha_2} - H_{\alpha_2}H_{\alpha_1} = \shin_{\alpha_1} H_{\alpha_2} - \shin_{\alpha_2} H_{\alpha_1} = \sum_{\alpha_1 \subset^{\shin}_{\alpha_2} \beta} \shin_{\beta} - \sum_{\alpha_2 \subset_{\alpha_1}^{\shin} \gamma} \shin_{\gamma}.$$ The set of $\beta$ such that $\alpha_1 \subset_{\alpha_2}^{\shin} \beta$ is given by either $\beta = (\beta_1, \beta_2)$ where $\beta_1 \geq \alpha_2$ and $\beta_2 = \alpha_1 + \alpha_2 - \beta_1$ or $\beta_1 = \alpha_1$ and $\beta_2 = \alpha_2$, due to the overhang rule. The set of $\gamma$ such that $\alpha_2 \subset_{\alpha_1}^{\shin} \gamma$ is given by $\gamma = (\gamma_1, \gamma_2)$ where $\gamma_1 \geq \alpha_1$ and $\gamma_2 = \alpha_1 + \alpha_2 - \gamma_1$. Then $\sum_{\alpha_1 \subset^{\shin}_{\alpha_2} \beta} \shin_{\beta} - \sum_{\alpha_2 \subset_{\alpha_1}^{\shin} \gamma} \shin_{\gamma} = \shin_{(\alpha_1, \alpha_2)}$. Thus, we have shown that $f_{\alpha} = \shin_{\alpha}$ when $\ell(\alpha)=2$.
\end{proof}

\begin{ex} Applying $\beth_2$ to the $H$-expansion of $\shin_{(3,1)}$ yields the $H$-expansion of $\shin_{(2,3,1)}$.
    $$\shin_{(3,1)} = H_{(3,1)} - H_{(4)} \text{\qquad and \qquad} \shin_{(2,3,1)} = \beth_{2}(\shin_{(3,1)}) = H_{(2,3,1)} - H_{(3,2,1)} - H_{(2,4)} + H_{(4,2)}.$$
\end{ex}

The creation operators also allow us to construct shin functions indexed by strictly increasing compositions from the ground up.\begin{cor}
    Let $\beta = (\beta_1, \ldots, \beta_k)$ be a strictly increasing sequence where $\beta_i < \beta_{i+1}$ for all $i$. Then, $$\shin_{\beta} = \beth_{\beta_1} \cdots \beth_{\beta_k}(1).$$
\end{cor}

\begin{ex}\label{shin_creation_op_ex} Creation operators can be used to build up a shin function as follows.
$$\shin_{(1,3,4)} = \beth_1 \beth_3 \beth_4 (1) = \beth_1 \beth_3 (H_{(4)}) = \beth_1 (H_{(3,4)} - H_{(4,3)}) = H_{(1,3,4)} - H_{(1,4,3)} - H_{(3,1,4)} + H_{(4,1,3)}.$$
\end{ex}

We also define a Jacobi-Trudi rule to express these same shin functions as matrix determinants. This formula is computationally much simpler and combinatorially more straightforward to work with. Let $S_{k}^{\geq}(-1)$ be the set of permutations $\sigma \in S_k$ such that $\sigma(i) \geq i-1$ for all $i \in [k]$.

\begin{thm} \label{JT_shin}
    Let $\beta = (\beta_1, \ldots, \beta_k)$ be a composition such that $\beta_i < \beta_{i+1}$ for all $i$. Then, $$\shin_{\beta} = \sum_{\sigma \in S_{k}^{\geq}(-1)} (-1)^{\sigma} H_{\beta_{\sigma(1)}} \cdots H_{\beta_{\sigma(k)}}.$$

    Equivalently, $\shin_{\beta}$ can be expressed as the matrix determinant of
    
    $$\shin_{\beta} = det \begin{vmatrix}
        H_{\beta_1} & H_{\beta_2} & H_{\beta_3}  & \cdots & H_{\beta_{k-2}} & H_{\beta_{k-1}} & H_{\beta_k}\\
        H_{\beta_1} & H_{\beta_2} & H_{\beta_3}  & \cdots & H_{\beta_{k-2}} &  H_{\beta_{k-1}} & H_{\beta_k}\\
        0 & H_{\beta_2} & H_{\beta_3}  & \cdots & H_{\beta_{k-2}} & H_{\beta_{k-1}} & H_{\beta_k}\\
        0 & 0 & H_{\beta_3}  &  \cdots & H_{\beta_{k-2}} & H_{\beta_{k-1}} & H_{\beta_k}\\
        \vdots & \vdots & \vdots & & \vdots & \vdots & \vdots \\
        0 & 0 & 0  &  \cdots & H_{\beta_{k-2}} & H_{\beta_{k-1}} & H_{\beta_k}\\
        0 & 0 & 0  & \cdots & H_{\beta_{k-2}} & H_{\beta_{k-1}} & H_{\beta_k}\\
        0 & 0 & 0  & \cdots & 0 & H_{\beta_{k-1}} & H_{\beta_k}
    \end{vmatrix}$$ \vspace{1mm}
    
    where we use the noncommutative analogue to the determinant obtained by expanding along the first row. 
\end{thm}

\begin{proof}
    We proceed by induction on the length of $\beta$. If $\ell(\beta)=1$ then $\shin_{\beta} = H_{\beta} = det|H_{\beta_1}|$, and our claim holds. Assume that our claim holds for all $\beta$ with $\ell(\beta) = k-1$.  Now consider $\beta = (\beta_1, \ldots, \beta_k)$. Since $\beta_1 < \beta_2$ we have $\shin_{\beta} = \beth_{\beta_1}(\shin_{\bar{\beta}})$ where $\bar{\beta} = (\bar{\beta}_1, \ldots, \bar{\beta}_{k-1}) = (\beta_2, \ldots, \beta_{k})$. By our inductive assumption, $$\shin_{\bar{\beta}} = \sum_{\pi \in S_{k-1}^{\geq}(-1)} (-1)^{\pi} H_{\bar{\beta}_{\pi(1)}} \cdots H_{\bar{\beta}_{\pi(k-1)}},$$ where the sum runs over $\pi$ such that $\pi(i)\geq i-1$ for all $i \in [k-1]$. Then, applying $\beth_{\beta_1}$ yields
    \begin{align*}
       \shin_{\beta} &= \sum_{\pi \in S_{k-1}^{\geq}(-1)} (-1)^{\pi}H_{\beta_1}H_{\bar{\beta}_{\pi(1)}} \cdots H_{\bar{\beta}_{\pi(k-1)}} - \sum_{\pi \in S_{k-1}^{\geq}(-1)} (-1)^{\pi}H_{\bar{\beta}_{\pi(1)}}H_{\beta_1}H_{\bar{\beta}_{\pi(2)}} \cdots H_{\bar{\beta}_{\pi(k-1)}}\\ \intertext{Now we rewrite our sums in terms of some subset of permutations $\sigma$ of $k$ and $\beta$ with terms that look like $H_{\beta_{\sigma(1)}}H_{\beta_{\sigma(2)}} \cdots H_{\beta_{\sigma(k)}}$. We need to find the appropriate subset of $S_k$ so that we correctly rewrite our sum.  For the first portion, we need $H_{\beta_{\sigma(1)}} = H_{\beta_1}$ so we set $\sigma(1)=1$.  Note that for all $i \in [k-1]$ we have $\bar{\beta}_i = \beta_{i+1}$.
       We also need $H_{\beta_{\sigma(i)}} = H_{\bar{\beta}_{\pi(i-1)}}$ which would mean $\sigma(i) = \pi(i-1) + 1$. Then $\sigma(i)\geq i-1$ for any $1<i \leq k$ because $\pi \in S_{k-1}^{\geq}(-1)$. Since we already set $\sigma(1) = 1 \geq 0$, we have shown that the permutations $\sigma$ satisfy $\sigma(i) \geq i-1$ for all $i$. For the right-hand term, we need $H_{\beta_{\sigma(2)}} = H_{\beta_1}$ so we limit our sum to permutations where $\sigma(2)=1$.  Note that we need $\beta_{\sigma(1)} = \bar{\beta}_{\pi(1)} = \beta_{\pi(1)+1}$ so $\sigma(1) \geq 2$, and $\beta_{\sigma(i)} = \bar{\beta}_{\pi(i-1)} = \beta_{\pi(i-1)+1}$ so $\sigma(i) \geq i-1$ for all $i \geq 3$. These are exactly the permutations $\sigma$ such that $\sigma(i) \geq i-1$ and $\sigma(2)=1$. Therefore,}
       \shin_{\beta} &= \sum_{\substack{\sigma \in S_{k}^{\geq}(-1),\\ \sigma(1)=1}} (-1)^{\sigma} H_{\beta_{\sigma(1)}} H_{\beta_{\sigma(2)}} \cdots H_{\beta_{\sigma(k)}} + \sum_{\substack{\sigma \in S^{\geq}_{k}(-1)\\ \sigma(2)=1}} (-1)^{\sigma} H_{\beta_{\sigma(1)}} H_{\beta_{\sigma(2)}} \cdots H_{\beta_{\sigma(k)}}\\
       &= \sum_{\sigma \in S_{k}^{\geq}(-1)} (-1)^{\sigma}H_{\beta_{\sigma(1)}}H_{\beta_{\sigma(2)}} \cdots H_{\beta_{\sigma(k)}}. 
    \end{align*}

    This is equivalent to the matrix determinant described using the typical expansion of determinants in terms of permutations, excluding those terms that would be multiplied by 0.
\end{proof}

\begin{ex} The function from Example \ref{shin_creation_op_ex} expands as a matrix determinant as
    $$\shin_{(1,3,4)} = det \begin{vmatrix} H_1 & H_3 & H_4 \\
    H_1 & H_3 & H_4 \\
    0 & H_3 & H_4
    \end{vmatrix} = H_1H_3H_4 - H_1H_4H_3 - H_3H_1H_4 + H_4H_1H_3.$$
\end{ex}

 It remains open to find a combinatorial or algebraic way of understanding the expansion of the shin basis into the complete homogeneous basis for the general case. We can show by counterexample that there is not a matrix rule of this form for every shin function, including those indexed by partitions. The smallest counterexample for partitions is $\shin_{(2,2,2,1)}$. The argument for it is too long, and so instead we present a smaller example using similar logic.

\begin{ex}
 We have $\shin_{(2,2,4)} = H_{(2, 2, 4)} - H_{(2, 4, 2)} - H_{(3, 1, 4)} + H_{(4, 3, 1)} + H_{(5, 1, 2)} - H_{(5, 2, 1)}$. For the determinant of a $3 \times 3$ matrix of the form $(H_{b_{(i,j)}})_{1 \leq i,j \leq 3}$ with $b_{(i,j)} \in \mathbb{Z}$ to yield this expression, we would need $H_2$, $H_3$, $H_4$, and $H_5$ to be in the first row. This is impossible, so such a matrix does not exist.
\end{ex}

\section{Skew and skew-II extended Schur functions}\label{section_skew}

\subsection{Skew extended Schur functions}

We define skew extended Schur functions first algebraically and then connect to tableaux combinatorics. For $H \in$ \textit{NSym}, the \emph{perp operator} $H^{\perp}$ acts on elements $F \in QSym$ based on the relation $\langle HG,F \rangle = \langle G,H^{\perp}F\rangle$. This expands as  
\begin{equation}\label{NSym_perp}
    H^{\perp}(F) = \sum_{\alpha}\langle HA_{\alpha},F \rangle B_{\alpha},
\end{equation} for dual bases $\{A_{\alpha}\}_{\alpha}$ of \textit{NSym} and $\{B_{\alpha}\}_{\alpha}$ of $QSym$.

\begin{defn}
For compositions $\alpha$ and $\beta$ with $\beta \subseteq \alpha$, the \emph{skew extended Schur functions} are defined as $$\shin^*_{\alpha/\beta} = \shin_{\beta}^{\perp} (\shin^*_{\alpha}).$$
\end{defn}

By Equation \eqref{NSym_perp}, the skew extended Schur function $\shin^*_{\alpha/\beta}$ expands into various bases as follows. For $\beta \subseteq \alpha$, 
\begin{equation}\label{prop_shin_skew_coeffs}
    \shin^*_{\alpha/\beta}= \sum_{\gamma} \langle \shin_{\beta}H_{\gamma}, \shin^*_{\alpha}  \rangle M_{\gamma} = \sum_{\gamma} \langle \shin_{\beta}R_{\gamma}, \shin^*_{\alpha} \rangle F_{\gamma} = \sum_{\gamma} \langle \shin_{\beta}\shin_{\gamma}, \shin^*_{\alpha} \rangle \shin^*_{\gamma}.
\end{equation}
 The coefficients $\mathcal{C}^{\alpha}_{\beta, \gamma} = \langle \shin_{\beta} \shin_{\gamma}, \shin^*_{\alpha}  \rangle$ are also the coefficients that appear when multiplying shin functions, $$\shin_{\beta} \shin_{\gamma} = \sum_{\alpha} \langle \shin_{\beta} \shin_{\gamma}, \shin^*_{\alpha}  \rangle \shin_{\alpha}.$$ Many of the shin structure coefficients are either 0 or equal to the Littlewood Richardson coefficients as a result of the shin basis' relationship with the Schur functions.

\begin{prop}\label{shin_struct_coeff}
     Let $\alpha, \beta,$ and $\gamma$ be compositions that are not partitions and let $\lambda, \mu,$ and $ \nu$ be partitions. Then, $$\mathcal{C}^{\lambda}_{\beta, \gamma} = \mathcal{C}^{\lambda}_{\beta, \nu} = \mathcal{C}^{\lambda}_{\mu, \gamma}  = 0 \text{\quad and \quad} \mathcal{C}^{\lambda}_{\mu, \nu} = c^{\lambda}_{\mu, \nu},$$ where $c^{\lambda}_{\mu, \nu}$ are the usual Littlewood-Richardson coefficients.
\end{prop}

\begin{proof} For any compositions $\alpha$ and $\beta$,  $$\shin_{\beta}\shin_{\gamma} = \sum_{\alpha} \mathcal{C}^{\alpha}_{\beta, \gamma} \shin_{\alpha}.$$
    Observe that if either $\beta$ or $\gamma$ is a composition that is not a partition then $$\chi(\shin_{\beta}\shin_{\gamma})=\chi(\shin_{\beta})\chi(\shin_{\gamma}) = 0 = \sum_{\alpha \models |\beta| + |\gamma|} \mathcal{C}^{\alpha}_{\beta, \gamma} \chi(\shin_{\alpha}) = \sum_{\lambda \vdash |\beta| + |\gamma|} \mathcal{C}^{\lambda}_{\beta, \gamma} s_{\lambda}.$$ 
    Thus, $\mathcal{C}^{\lambda}_{\beta, \gamma} = 0$ for all partitions $\lambda$ when either $\beta$ or $\gamma$ is not a partition. If $\beta$ and $\gamma$ are both partitions, which we instead call $\mu$ and $\nu$, we have $$\chi(\shin_{\mu} \shin_{\nu}) = \chi(\shin_{\mu})\chi(\shin_{\nu}) = s_{\mu}s_{\nu} = \sum_{\alpha \models  |\mu| + |\nu| } \mathcal{C}^{\alpha}_{\mu, \nu} \chi(\shin_{\alpha}) = \sum_{\lambda \vdash |\mu| + |\nu|} \mathcal{C}^{\lambda}_{\mu, \nu} s_{\lambda},$$
    and so $\mathcal{C}^{\lambda}_{\mu, \nu} = c^{\lambda}_{\mu, \nu}$, the classic Littlewood-Richardson coefficients, for partitions $\lambda$.
\end{proof}

The skew shin functions are defined combinatorially over skew shin-tableaux as follows.

\begin{prop}\label{skew_into_tableaux}
    For compositions $\alpha$ and $\beta$ such that $\beta \subseteq \alpha$, the coefficient $\langle \shin_{\beta} H_{\gamma}, \shin^*_{\alpha} \rangle$ is equal to the number of skew shin-tableaux of shape $\alpha/\beta$ and type $\gamma$. Moreover, $$\shin^*_{\alpha/\beta} = \sum_T x^T,$$ where the sum runs over skew shin-tableaux $T$ of shape $\alpha/\beta$. 
\end{prop}

\begin{proof}
Observe that $$\shin_{\beta}H_{\gamma} = \sum_{\beta = \beta^{(0)} \subset^{\shin}_{\gamma_1} \beta^{(1)} \subset^{\shin}_{\gamma_2} \beta^{(2)} \cdots \subset^{\shin}_{\gamma_j} \beta^{(j)} = \alpha} \shin_{\alpha}$$ via repeated applications of the right Pieri rule.  
    Then, $\langle \shin_{\beta} H_{\gamma}, \shin_{\alpha}^* \rangle$ counts the number of unique sequences $(\beta_0, \beta_1, \ldots, \beta_j)$ such that $\beta = \beta^{(0)} \subset^{\shin}_{\gamma_1} \beta^{(1)} \subset^{\shin}_{\gamma_2} \beta^{(2)} \cdots \subset^{\shin}_{\gamma_j} \beta^{(j)} = \alpha$. Each of these sequences can be associated with a unique skew shin-tableau of shape $\alpha/\beta$ with type $\gamma$ by filling with $i's$ the blocks in $\beta^{(i)} / \beta^{(i-1)}$.  The containment condition of the Pieri rule ensures that rows are weakly increasing and the overhang rule ensures that columns are strictly increasing. The overhang rule also ensures that if $\alpha_i> \beta_i$ for any  $1 \leq i \leq k$, then $\beta_j < \beta_i$ for all $j > i$. Thus, our tableau $T$ is indeed a skew shin-tableaux.  It is simple to see that any skew shin-tableau can be expressed by such a sequence, showing that $\langle \shin_{\beta} H_{\gamma}, \shin^*_{\alpha} \rangle$ is equal to the number of skew shin-tableaux of shape $\alpha/\beta$ and type $\gamma$.  
    It now follows from Equation \eqref{prop_shin_skew_coeffs} that $\shin^*_{\alpha/\beta} = \sum_{T} x^T$ where the sum runs over all skew shin-tableaux of shape $\alpha/\beta$.
\end{proof}

\begin{ex} The skew extended Schur function indexed by $(3,4)/(2,1)$ is given by

 $$\scalebox{.85}{
    \begin{ytableau}
        *(gray) & *(gray) & 1 \\
        *(gray) & 1 & 2 & 2  
    \end{ytableau}
    \qquad
    \begin{ytableau}
        *(gray) & *(gray) & 1 \\
        *(gray) & 1 & 2 & 3
    \end{ytableau}
     \qquad
    \begin{ytableau}
        *(gray) & *(gray) & 1 \\
        *(gray) & 2 & 3 & 4
    \end{ytableau}
     \qquad
    \begin{ytableau}
        *(gray) & *(gray) & 2 \\
        *(gray) & 1 & 3 & 4
    \end{ytableau}
     \qquad
    \begin{ytableau}
        *(gray) & *(gray) & 2  \\
        *(gray) & 3 & 3 & 3
    \end{ytableau} } \quad \cdots
   $$

    $$\shin^*_{(3,4)/(2,1)} = x_1^2x_2^2 + x_1^2x_2x_3 + x_1x_2x_3x_4 + x_1x_2x_3x_4 + x_2x_3^3 + \ldots $$
\end{ex}
   
Skew shin-tableaux of shape $\lambda/\mu$ for partitions $\lambda, \mu$ are skew semistandard Young tableaux and thus by Proposition \ref{skew_into_tableaux}, the skew extended Schur functions are equal to the skew Schur functions. That is, for partitions $\lambda$ and $\mu$ where $\mu \subseteq \lambda$, $$\shin^*_{\lambda/\mu} = s_{\lambda/\mu}.$$

\begin{rem}
    It also follows from Proposition \ref{skew_into_tableaux} that $\mathcal{C}^{\alpha}_{\beta, \gamma} = 0$ if there exists some $i$ and $j$ with $i < j$ where $\alpha_i > \beta_i$ but $\beta_j > \beta_i$. In particular, this tells us that many terms that do not appear in  $\shin_{\beta}\shin_{\gamma}$. 
\end{rem}

Standard skew shin-tableaux are also closely related to the poset structure on composition diagrams whose cover relations are determined by the addition of a shin-horizontal-strip of size $1$. 

\begin{defn} Define the \emph{shin poset}, denoted $\mathcal{P}^{\shin}$, as the poset on compositions with the cover relation $\subset^{\shin}_1$. In the Hasse diagram, label an edge between two elements $\alpha$ and $\beta$ with the integer $m$ if $\alpha$ differs from $\beta$ by the addition of a box to row $m$.
\end{defn}
\tikzset{every picture/.style={line width=0.75pt}} 
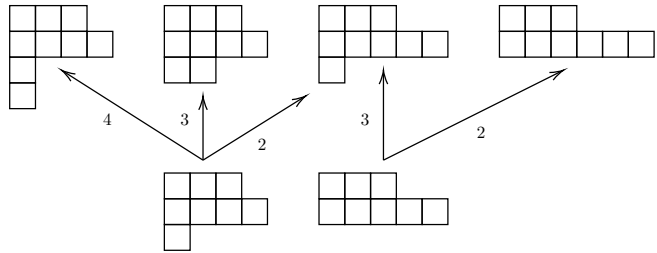
\begin{figure}[h!]
\scalebox{.65}{
\begin{tikzpicture}[x=0.75pt,y=0.75pt,yscale=-1,xscale=1]

\draw   (140,190) -- (160,190) -- (160,210) -- (140,210) -- cycle ;
\draw   (140,170) -- (160,170) -- (160,190) -- (140,190) -- cycle ;
\draw   (160,170) -- (180,170) -- (180,190) -- (160,190) -- cycle ;
\draw   (180,170) -- (200,170) -- (200,190) -- (180,190) -- cycle ;
\draw   (200,170) -- (220,170) -- (220,190) -- (200,190) -- cycle ;
\draw   (140,150) -- (160,150) -- (160,170) -- (140,170) -- cycle ;
\draw   (160,150) -- (180,150) -- (180,170) -- (160,170) -- cycle ;
\draw   (180,150) -- (200,150) -- (200,170) -- (180,170) -- cycle ;
\draw   (260,60) -- (280,60) -- (280,80) -- (260,80) -- cycle ;
\draw   (260,40) -- (280,40) -- (280,60) -- (260,60) -- cycle ;
\draw   (280,40) -- (300,40) -- (300,60) -- (280,60) -- cycle ;
\draw   (300,40) -- (320,40) -- (320,60) -- (300,60) -- cycle ;
\draw   (320,40) -- (340,40) -- (340,60) -- (320,60) -- cycle ;
\draw   (260,20) -- (280,20) -- (280,40) -- (260,40) -- cycle ;
\draw   (280,20) -- (300,20) -- (300,40) -- (280,40) -- cycle ;
\draw   (300,20) -- (320,20) -- (320,40) -- (300,40) -- cycle ;
\draw   (140,60) -- (160,60) -- (160,80) -- (140,80) -- cycle ;
\draw   (140,40) -- (160,40) -- (160,60) -- (140,60) -- cycle ;
\draw   (160,40) -- (180,40) -- (180,60) -- (160,60) -- cycle ;
\draw   (180,40) -- (200,40) -- (200,60) -- (180,60) -- cycle ;
\draw   (200,40) -- (220,40) -- (220,60) -- (200,60) -- cycle ;
\draw   (140,20) -- (160,20) -- (160,40) -- (140,40) -- cycle ;
\draw   (160,20) -- (180,20) -- (180,40) -- (160,40) -- cycle ;
\draw   (180,20) -- (200,20) -- (200,40) -- (180,40) -- cycle ;
\draw   (20,60) -- (40,60) -- (40,80) -- (20,80) -- cycle ;
\draw   (20,40) -- (40,40) -- (40,60) -- (20,60) -- cycle ;
\draw   (40,40) -- (60,40) -- (60,60) -- (40,60) -- cycle ;
\draw   (60,40) -- (80,40) -- (80,60) -- (60,60) -- cycle ;
\draw   (80,40) -- (100,40) -- (100,60) -- (80,60) -- cycle ;
\draw   (20,20) -- (40,20) -- (40,40) -- (20,40) -- cycle ;
\draw   (40,20) -- (60,20) -- (60,40) -- (40,40) -- cycle ;
\draw   (60,20) -- (80,20) -- (80,40) -- (60,40) -- cycle ;
\draw   (340,40) -- (360,40) -- (360,60) -- (340,60) -- cycle ;
\draw   (160,60) -- (180,60) -- (180,80) -- (160,80) -- cycle ;
\draw   (20,80) -- (40,80) -- (40,100) -- (20,100) -- cycle ;
\draw   (260,170) -- (280,170) -- (280,190) -- (260,190) -- cycle ;
\draw   (280,170) -- (300,170) -- (300,190) -- (280,190) -- cycle ;
\draw   (300,170) -- (320,170) -- (320,190) -- (300,190) -- cycle ;
\draw   (320,170) -- (340,170) -- (340,190) -- (320,190) -- cycle ;
\draw   (260,150) -- (280,150) -- (280,170) -- (260,170) -- cycle ;
\draw   (280,150) -- (300,150) -- (300,170) -- (280,170) -- cycle ;
\draw   (300,150) -- (320,150) -- (320,170) -- (300,170) -- cycle ;
\draw   (340,170) -- (360,170) -- (360,190) -- (340,190) -- cycle ;
\draw   (500,40) -- (520,40) -- (520,60) -- (500,60) -- cycle ;
\draw   (400,40) -- (420,40) -- (420,60) -- (400,60) -- cycle ;
\draw   (420,40) -- (440,40) -- (440,60) -- (420,60) -- cycle ;
\draw   (440,40) -- (460,40) -- (460,60) -- (440,60) -- cycle ;
\draw   (460,40) -- (480,40) -- (480,60) -- (460,60) -- cycle ;
\draw   (400,20) -- (420,20) -- (420,40) -- (400,40) -- cycle ;
\draw   (420,20) -- (440,20) -- (440,40) -- (420,40) -- cycle ;
\draw   (440,20) -- (460,20) -- (460,40) -- (440,40) -- cycle ;
\draw   (480,40) -- (500,40) -- (500,60) -- (480,60) -- cycle ;
\draw    (170,140) -- (61.69,71.07) ;
\draw [shift={(60,70)}, rotate = 32.47] [color={rgb, 255:red, 0; green, 0; blue, 0 }  ][line width=0.75]    (10.93,-3.29) .. controls (6.95,-1.4) and (3.31,-0.3) .. (0,0) .. controls (3.31,0.3) and (6.95,1.4) .. (10.93,3.29)   ;
\draw    (170,140) -- (170,92) ;
\draw [shift={(170,90)}, rotate = 90] [color={rgb, 255:red, 0; green, 0; blue, 0 }  ][line width=0.75]    (10.93,-3.29) .. controls (6.95,-1.4) and (3.31,-0.3) .. (0,0) .. controls (3.31,0.3) and (6.95,1.4) .. (10.93,3.29)   ;
\draw    (170,140) -- (248.3,91.06) ;
\draw [shift={(250,90)}, rotate = 147.99] [color={rgb, 255:red, 0; green, 0; blue, 0 }  ][line width=0.75]    (10.93,-3.29) .. controls (6.95,-1.4) and (3.31,-0.3) .. (0,0) .. controls (3.31,0.3) and (6.95,1.4) .. (10.93,3.29)   ;
\draw    (310,140) -- (310,72) ;
\draw [shift={(310,70)}, rotate = 90] [color={rgb, 255:red, 0; green, 0; blue, 0 }  ][line width=0.75]    (10.93,-3.29) .. controls (6.95,-1.4) and (3.31,-0.3) .. (0,0) .. controls (3.31,0.3) and (6.95,1.4) .. (10.93,3.29)   ;
\draw    (310,140) -- (448.21,70.89) ;
\draw [shift={(450,70)}, rotate = 153.43] [color={rgb, 255:red, 0; green, 0; blue, 0 }  ][line width=0.75]    (10.93,-3.29) .. controls (6.95,-1.4) and (3.31,-0.3) .. (0,0) .. controls (3.31,0.3) and (6.95,1.4) .. (10.93,3.29)   ;

\draw (91,102.4) node [anchor=north west][inner sep=0.75pt]    {$4$};
\draw (151,102.4) node [anchor=north west][inner sep=0.75pt]    {$3$};
\draw (211,122.4) node [anchor=north west][inner sep=0.75pt]    {$2$};
\draw (291,102.4) node [anchor=north west][inner sep=0.75pt]    {$3$};
\draw (381,112.4) node [anchor=north west][inner sep=0.75pt]    {$2$};

\end{tikzpicture}
}
\caption{A section of the shin poset $\mathcal{P}^{\shin}$.}
\end{figure} 

\noindent That is to say, $\beta$ covers $\alpha$ in the shin poset if $\beta$ differs from $\alpha$ by the addition of a shin-horizontal strip of size 1 (a single box).  Note that the subposet of $\mathcal{P}^{\shin}$ induced on partitions corresponds to Young's Lattice. Further, we generally visualize the elements of this poset as the diagrams associated with each composition.

Maximal chains from $\beta$ to $\alpha$ in $\mathcal{P}^{\shin}$ can be associated with a skew standard shin-tableau of shape $\alpha/\beta$. The chain $C = \{\beta = \beta^{(0)} \xrightarrow{m_1} \beta^{(1)} \xrightarrow{m_2} \cdots \xrightarrow{m_k} \beta^{(k)} = \alpha \}$ is associated with the skew standard shin-tableau $T$ of shape $\alpha/\beta$ where the boxes are filled with the integers $1$ through $k$ in the order they are added on the chain.  Thus, the box added from $\beta^{(j)} \xrightarrow{m_{j+1}} \beta^{(j+1)}$ is filled with $j+1$. If $\beta = \emptyset$ then the chain $C$ is associated with a standard shin-tableau that is not skew.

\begin{ex}
    The maximal chain $C = \{ (2,1) \xrightarrow{1} (3,1) \xrightarrow{3} (3,1,1) \xrightarrow{2} (3,2,1) \xrightarrow{2} (3,3,1) \xrightarrow{2} (3,4,1)\}$ is associated with the skew standard shin-tableau:\vspace{2mm}
    \center{\scalebox{.85}{
    \begin{ytableau}
        *(gray) & *(gray) & 1 \\
        *(gray) & 3 & 4 & 5 \\
        2
    \end{ytableau}}}\vspace{2mm}
\end{ex}

As previously mentioned, $QSym$ admits a Hopf algebra structure \cite{doliwa21}. This means it has an operation called \emph{comultiplication} which we can now express for extended Schur functions.

\begin{prop} For a composition $\alpha$, $$\Delta \shin^*_{\alpha} = \sum_{\beta \subseteq \alpha} \shin^*_{\beta} \otimes \shin^*_{\alpha/\beta},$$ where the sum runs over compositions $\beta$.
\end{prop}

\begin{proof}
    The product of the shin functions uniquely defines the coproduct of the extended Schur functions due to Hopf algebra properties \cite{grinberg}. Specifically, 
    \begin{align*}
        \Delta(\shin_{\alpha}^*) &= \sum_{\beta, \gamma} \langle \shin_{\beta} \shin_{\gamma}, \shin^*_{\alpha} \rangle \shin_{\beta}^* \otimes \shin_{\gamma}^*  = \sum_{\beta} \left( \shin_{\beta}^* \otimes  \sum_{\gamma} \langle \shin_{\beta} \shin_{\gamma}, \shin^*_{\alpha} \rangle \shin^*_{\gamma} \right)\\
        &= \sum_{\beta} \shin^*_{\beta} \otimes \shin^*_{\alpha/\beta} \text{\quad by Equation \eqref{prop_shin_skew_coeffs}.}\qedhere
    \end{align*}
\end{proof}

\subsection{Skew-II extended Schur functions}

Because \textit{NSym} is noncommutative, there are two reasonable ways to define `skew' functions in $QSym$.  For our second type of skew function, we introduce a variation on the perp operator of Equation \eqref{NSym_perp}.

\begin{defn}\label{NSym_rperp}
For $H \in$ \textit{NSym}, the \emph{right-perp operator} $H^{\rperp}$ acts on elements $F \in QSym$ based on the relation $\langle GH,F \rangle = \langle G,H^{\rperp}F\rangle$. This expands as  $$H^{\rperp}(F) = \sum_{\alpha}\langle A_{\alpha}H,F \rangle B_{\alpha},$$ for dual bases $\{A_{\alpha}\}_{\alpha}$ of \textit{NSym} and $\{B_{\alpha}\}_{\alpha}$ of $QSym$.
\end{defn}

The usual perp operator $H^{\perp}$ of $QSym$ is dual to \emph{left} multiplication by $H$ in \textit{NSym}, whereas the right-perp operator $H^{\rperp}$ of $QSym$ is dual to \emph{right} multiplication by $H$ in \textit{NSym}. Because $Sym$ is commutative and self-dual, the perp and right-perp operators for $Sym$ are equivalent.

\begin{rem} In \cite{Nie2022pieri} and \cite{Tew}, left and right actions of \textit{NSym} on $QSym$ are used in place of the perp operators. Our $\shin_{\beta}^{\rperp}(\shin^*_{\alpha})$ is equivalent to their left action of \textit{NSym} on $QSym$. 
\end{rem}

Before introducing our alternative skew functions, we define the objects that serve as their indices.

\begin{defn} For compositions $\alpha = (\alpha_1, \ldots, \alpha_k)$ and $\beta = (\beta_1, \ldots, \beta_j)$ such that $\beta^r \subseteq \alpha^r$, the \emph{\otherskew\  shape} $\alpha \rskew \beta$ is the composition diagram of $\alpha$ where the left-most $\beta_i$ boxes are removed from row $\alpha_{k+1-i}$ for $1 \leq i \leq j$. This removal is often represented by shading the boxes in.
\end{defn}

\begin{ex} The following diagram is the \otherskew\ shape $(3,4,4)\rskew(2,1)$.

$$\scalebox{.85}{ 
\begin{ytableau}
     \ & \ & \  \\
     *(gray) & *(gray) & \ & \ \\
     *(gray) & \  & \ & \ 
\end{ytableau}}
$$
\end{ex}

This notation and concept come from \cite{Luoto2013}, but we use a different name here for clarity and continuity within this paper. Now, we define functions using the skew-II shapes.

\begin{defn}
    For compositions $\alpha$ and $\beta$, the \emph{skew-II extended Schur function} is defined as $$ \shin^*_{\alpha\rskew\beta} = \shin^{\rperp}_{\beta}(\shin^*_{\alpha}).$$
\end{defn}

By Definition \ref{NSym_rperp}, $\shin^*_{\alpha \rskew \beta}$ expands into various bases as follows. For compositions $\alpha$ and $\beta$, 

\begin{equation}\label{skewii_eschur_exp}
\shin^*_{\alpha\rskew\beta} = \sum_{\gamma} \langle H_{\gamma} \shin_{\beta}, \shin^*_{\alpha} \rangle M_{\gamma}  = \sum_{\gamma} \langle R_{\gamma} \shin_{\beta}, \shin^*_{\alpha} \rangle F_{\gamma}  = \sum_{\gamma} \langle \shin_{\gamma} \shin_{\beta}, \shin^*_{\alpha} \rangle \shin^*_{\gamma}.     
\end{equation}

In terms of the shin structure coefficients, we have $\shin^*_{\alpha\rskew\beta} = \sum_{\gamma} \mathcal{C}^{\alpha}_{\gamma, \beta} \shin^*_{\gamma}$. 

\begin{rem}
    According to calculations done in Sagemath \cite{sagemath}, the skew-II extended Schur function $\shin^*_{\alpha \rskew \beta}$ does not expand positively into the monomial basis. For example, $\shin^*_{(2,1,3) \rskew (1,2,1)}$ has the term $- M_{(1,1)}$ in its expansion. Thus, these functions cannot be expressed as positive sums of a skew-II shin-tableaux.
\end{rem}

The comultiplication of the extended Schur basis can also be expressed in terms of the \otherskew\ extended Schur functions. Note that a similar formula could be given for any basis using analogous skew-II functions defined via the perp operator.

\begin{prop}
    For a composition $\alpha$, $$\Delta(\shin^*_{\alpha}) = \sum_{\beta} \shin^*_{\alpha \rskew \beta} \otimes \shin^*_{\beta},$$ where the sum runs over compositions $\beta$ where $\beta^r \subseteq \alpha^r$.
\end{prop}

\begin{proof} For a composition $\alpha$,
\begin{align*} 
        \Delta(\shin_{\alpha}^*) &= \sum_{\beta, \gamma} \langle \shin_{\gamma} \shin_{\beta}, \shin^*_{\alpha} \rangle \shin_{\gamma}^* \otimes \shin_{\beta}^*  = \sum_{\beta} \left( \sum_{\gamma} \langle \shin_{\gamma} \shin_{\beta}, \shin^*_{\alpha} \rangle \shin^*_{\gamma} \right) \otimes \shin^*_{\beta}= \sum_{\beta} \shin^*_{\alpha\rskew \beta} \otimes \shin^*_{\beta}. \qedhere
    \end{align*}
    where the sum runs over compositions $\beta$.
\end{proof}

\section{Involutions on $QSym$ and \textit{NSym}}\label{section_invol}

We consider three involutions in $QSym$ defined on the fundamental basis that correspond with three involutions in \textit{NSym} defined on the ribbon basis \cite{Luoto2013}. All six maps are defined as extensions of the involutions given by the complement, reverse, and transpose operations on compositions. 
 

 \begin{defn}
     The involutions $\psi$, $\rho$, and $\omega$ on $QSym$ and \textit{NSym} are defined as
      $$ \psi(F_{\alpha}) = F_{\alpha^c} \quad \quad \rho(F_{\alpha}) = F_{\alpha^r} \quad \quad \omega(F_{\alpha})=F_{\alpha^t},$$ $$\psi(R_{\alpha}) = R_{\alpha^c}
 \quad \quad \rho(R_{\alpha}) = R_{\alpha^r} \quad \quad  \omega(R_{\alpha})=R_{\alpha^t},$$ extended linearly. All three maps on $QSym$ and $\psi$ on \textit{NSym} are automorphisms, while $\rho$ and $\omega$ on \textit{NSym} are anti-automorphisms.
 \end{defn}

 Note that we use the same notation for the corresponding involutions on $QSym$ and \textit{NSym}.  These maps commute, and $\omega = \rho \circ \psi = \psi \circ \rho$. When $\omega$ and $\psi$ are restricted to $Sym$, they are both equivalent to the classical involution $\omega: Sym \rightarrow Sym$ which acts on the Schur functions by $\omega(s_{\lambda}) = \lambda'$ where $\lambda'$ is the conjugate of $\lambda$. Additionally, Jia, Wang, and Yu prove in \cite{jia2019rigidity} that $\psi$, $\rho$, and $\omega$ are in fact the only nontrivial graded algebra automorphisms on $QSym$ that preserve the fundamental basis. Further, $\psi$ is the only nontrivial graded Hopf algebra automorphism on $QSym$ that preserves the fundamental basis.

We introduce two new pairs of dual bases ($\fshin$, $\fshin^*$ and $\bshin$, $\bshin^*$) in $QSym$ and \textit{NSym} by applying $\rho$ and $\omega$ to the extended Schur and shin functions. Applying $\psi$ to the extended Schur and shin functions recovers the row-strict shin and row-strict extended Schur functions ($\rshin$, $\rshin^*$) of Niese, Sundaram, van Willigenburg, Vega, and Wang from \cite{row0hecke}. Specifically, for a composition $\alpha$, we have
$$\psi(\shin^*_{\alpha}) = \rshin^*_{\alpha} \quad \quad \rho(\shin^*_{\alpha}) = \fshin^*_{\alpha^r} \quad \quad \omega(\shin^*_{\alpha}) = \bshin^*_{\alpha^r},$$
$$\psi(\shin_{\alpha}) = \rshin_{\alpha} \quad \quad \rho(\shin_{\alpha}) = \fshin_{\alpha^r} \quad \quad \omega(\shin_{\alpha}) = \bshin_{\alpha^r}.$$

We then give combinatorial interpretations of these 2 new pairs of bases in terms of variations on shin-tableaux. Recall that shin-tableaux have weakly increasing columns and strictly increasing rows.  Intuitively, the $\psi$ map switches whether the strictly changing condition is on rows or columns (the other is allowed to change weakly).  The $\rho$ map switches the row condition from increasing to decreasing or vice versa. The $\omega$ map does both.  Table \ref{table_shin} summarizes the tableaux defined throughout this section with the position of $i+1$ relative to $i$ that makes $i$ a descent, the order of the reading word (Left, Right, Top, Bottom), the condition on entries of each row, and the condition on entries in each column.
\begin{center}
\begin{table}[h!]
\begin{tabular}{ |c|c|c|c|c| }
 \hline
   & \textbf{Descent} & \textbf{Reading Word} 
 & \textbf{Rows} & \textbf{Columns}\\ 
 \hline
 \textbf{Shin} & strictly below & L to R, B to T & weakly increasing & strictly increasing \\ 
 \hline
 \textbf{Row-strict}  & weakly above & L to R, T to B & strictly increasing & weakly increasing \\ 
 \hline
 \textbf{Reverse} & strictly below & R to L, B to T & weakly decreasing  & strictly increasing \\ 
 \hline
 \textbf{Row-strict reverse} & weakly above & R to L, T to B & strictly decreasing & weakly increasing\\ 
 \hline
\end{tabular}
\caption{Variations on shin tableaux.}\label{table_shin}
\end{table}
\end{center}

Through this combinatorial interpretation, each of the four pairs of dual bases is related to any other by one of the three involutions $\psi$, $\rho$, or $\omega$ as shown in Figure \ref{fig:involutions}. Recall that the involution $\psi$, $\rho$, and $\omega$ in $QSym$ and \textit{NSym} collectively serve as the analogue to $\omega$ in $Sym$. In $Sym$, the Schur basis is its own image under $\omega$ but in $QSym$ and \textit{NSym} our Schur-like bases are instead part of a system of 4 related bases that is closed with respect to $\psi, \rho$ and $\omega$.

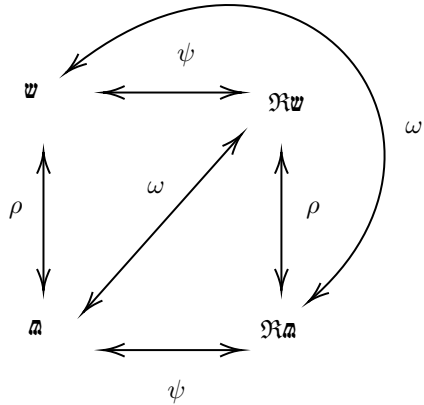
\begin{figure}[h!]\label{involution_diagram}
    \centering
\tikzset{every picture/.style={line width=0.75pt}} 
\qquad \quad \begin{tikzpicture}[x=0.75pt,y=0.75pt,yscale=-1,xscale=1]

\draw    (25,81) -- (25,147) ;
\draw [shift={(25,149)}, rotate = 270] [color={rgb, 255:red, 0; green, 0; blue, 0 }  ][line width=0.75]    (10.93,-3.29) .. controls (6.95,-1.4) and (3.31,-0.3) .. (0,0) .. controls (3.31,0.3) and (6.95,1.4) .. (10.93,3.29)   ;
\draw [shift={(25,79)}, rotate = 90] [color={rgb, 255:red, 0; green, 0; blue, 0 }  ][line width=0.75]    (10.93,-3.29) .. controls (6.95,-1.4) and (3.31,-0.3) .. (0,0) .. controls (3.31,0.3) and (6.95,1.4) .. (10.93,3.29)   ;
\draw    (38.08,37.56) .. controls (153.12,-56.21) and (255.13,73.12) .. (158.5,154.71) ;
\draw [shift={(158.5,154.71)}, rotate = 319.77] [color={rgb, 255:red, 0; green, 0; blue, 0 }  ][line width=0.75]    (10.93,-3.29) .. controls (6.95,-1.4) and (3.31,-0.3) .. (0,0) .. controls (3.31,0.3) and (6.95,1.4) .. (10.93,3.29)   ;
\draw [shift={(36.33,39)}, rotate = 320.18] [color={rgb, 255:red, 0; green, 0; blue, 0 }  ][line width=0.75]    (10.93,-3.29) .. controls (6.95,-1.4) and (3.31,-0.3) .. (0,0) .. controls (3.31,0.3) and (6.95,1.4) .. (10.93,3.29)   ;
\draw    (57,49) -- (123,49) ;
\draw [shift={(125,49)}, rotate = 180] [color={rgb, 255:red, 0; green, 0; blue, 0 }  ][line width=0.75]    (10.93,-3.29) .. controls (6.95,-1.4) and (3.31,-0.3) .. (0,0) .. controls (3.31,0.3) and (6.95,1.4) .. (10.93,3.29)   ;
\draw [shift={(55,49)}, rotate = 0] [color={rgb, 255:red, 0; green, 0; blue, 0 }  ][line width=0.75]    (10.93,-3.29) .. controls (6.95,-1.4) and (3.31,-0.3) .. (0,0) .. controls (3.31,0.3) and (6.95,1.4) .. (10.93,3.29)   ;
\draw    (145,81) -- (145,147) ;
\draw [shift={(145,149)}, rotate = 270] [color={rgb, 255:red, 0; green, 0; blue, 0 }  ][line width=0.75]    (10.93,-3.29) .. controls (6.95,-1.4) and (3.31,-0.3) .. (0,0) .. controls (3.31,0.3) and (6.95,1.4) .. (10.93,3.29)   ;
\draw [shift={(145,79)}, rotate = 90] [color={rgb, 255:red, 0; green, 0; blue, 0 }  ][line width=0.75]    (10.93,-3.29) .. controls (6.95,-1.4) and (3.31,-0.3) .. (0,0) .. controls (3.31,0.3) and (6.95,1.4) .. (10.93,3.29)   ;
\draw    (57,179) -- (123,179) ;
\draw [shift={(125,179)}, rotate = 180] [color={rgb, 255:red, 0; green, 0; blue, 0 }  ][line width=0.75]    (10.93,-3.29) .. controls (6.95,-1.4) and (3.31,-0.3) .. (0,0) .. controls (3.31,0.3) and (6.95,1.4) .. (10.93,3.29)   ;
\draw [shift={(55,179)}, rotate = 0] [color={rgb, 255:red, 0; green, 0; blue, 0 }  ][line width=0.75]    (10.93,-3.29) .. controls (6.95,-1.4) and (3.31,-0.3) .. (0,0) .. controls (3.31,0.3) and (6.95,1.4) .. (10.93,3.29)   ;
\draw    (123.67,71.49) -- (46.33,158.51) ;
\draw [shift={(45,160)}, rotate = 311.63] [color={rgb, 255:red, 0; green, 0; blue, 0 }  ][line width=0.75]    (10.93,-3.29) .. controls (6.95,-1.4) and (3.31,-0.3) .. (0,0) .. controls (3.31,0.3) and (6.95,1.4) .. (10.93,3.29)   ;
\draw [shift={(125,70)}, rotate = 131.63] [color={rgb, 255:red, 0; green, 0; blue, 0 }  ][line width=0.75]    (10.93,-3.29) .. controls (6.95,-1.4) and (3.31,-0.3) .. (0,0) .. controls (3.31,0.3) and (6.95,1.4) .. (10.93,3.29)   ;

\draw (137,47.4) node [anchor=north west][inner sep=0.75pt]  {$\rshin$};
\draw (16,162.4) node [anchor=north west][inner sep=0.75pt]   {$\fshin$};
\draw (133,162.4) node [anchor=north west][inner sep=0.75pt]  {$\bshin$};
\draw (14,42.4) node [anchor=north west][inner sep=0.75pt]  {$\shin$};
\draw (91,22.4) node [anchor=north west][inner sep=0.75pt]    {$\psi$};
\draw (86,192.4) node [anchor=north west][inner sep=0.75pt]    {$\psi$};
\draw (76,92.4) node [anchor=north west][inner sep=0.75pt]    {$\omega$};
\draw (206,62.4) node [anchor=north west][inner sep=0.75pt]    {$\omega$};
\draw (6,102.4) node [anchor=north west][inner sep=0.75pt]    {$\rho$};
\draw (156,102.4) node [anchor=north west][inner sep=0.75pt]    {$\rho$};

\end{tikzpicture}

    \caption{Mappings between shin variants in \textit{NSym}.}
    \label{fig:involutions}
\end{figure}

First, we review the row-strict extended Schur and shin functions and present a few new results. Then, we introduce the reverse extended Schur basis and row-strict reverse extended Schur basis in $QSym$, as well as their dual bases in \textit{NSym}. Using the three involutions, we translate many results from the shin and extended Schur functions to our new bases, including the Jacobi-Trudi rule and two different types of skew functions. We also note connections between the row-strict reverse extended Schur basis and the antipode of $QSym$ on the extended Schur functions, as well as the dual connection in \textit{NSym}. 

\subsection{Row-strict extended Schur and shin functions}

The row-strict extended Schur and row-strict shin bases were introduced by Niese, Sundaram, van Willigenburg, Vega, and Wang in \cite{row0hecke}, motivated by the representation theory of $0$-Hecke modules. Let $\alpha$ be a composition and let $\beta$ be a weak composition.  A \emph{row-strict shin-tableau} (RSST) of shape $\alpha$ and type $\beta$ is a filling of the composition diagram of $\alpha$ with positive integers such that each row strictly increases from left to right, each column weakly increases from top to bottom, and each integer $i$ appears $\beta_i$ times. A \emph{standard} row-strict shin-tableau (SRSST) with $n$ boxes is one containing the entries $1$ through $n$ each exactly once. 

\begin{defn}
For a composition $\alpha$, define the \emph{row-strict extended Schur function} as $$\rshin^*_{\alpha} = \sum_T x^T,$$ where the sum runs over all row-strict shin-tableaux $T$ of shape $\alpha$. The \emph{row-strict shin functions} $\rshin_{\alpha}$ are defined as the duals in \textit{NSym} to the row-strict extended Schur functions in $QSym$.
\end{defn}

The \emph{descent} set is defined to be $Des_{\rshin}(U) = \{ i : \text{$i+1$ is weakly above $i$ in $U$}\}$ for a standard row-strict shin-tableau $U$. Each entry $i$ in $Des_{\rshin}(U)$ is called a \emph{descent} of $U$. The \emph{descent composition} of $U$ is defined to be  $co_{\rshin}(U) = (i_1, i_2 - i_1, \ldots, i_{d}-i_{d-1}, n-i_{d})$ for $Des_{\rshin}(U) = \{i_1, \ldots, i_d\}$. Note that the set of standard row-strict shin-tableaux is exactly the same as the set of standard shin-tableaux.

Using the framework of standard row-strict shin-tableaux, it is shown in \cite{row0hecke} that for a composition $\alpha$, the row-strict extended Schur function expands into the fundamental basis as 
\begin{equation}\label{rshin*_to_F}
    \rshin^*_{\alpha} = \sum_{U} F_{co_{\rshin}(U)},
\end{equation} where the sum runs over all standard row-strict shin-tableaux.

\begin{ex} The $F$-expansion of the row-strict extended Schur function $\rshin^*_{(2,3)}$ and the standard row-strict shin-tableaux of shape $(2,3)$ are:
$$\rshin^*_{(2,3)} = F_{(1,2,1,1)} + F_{(2,2,1)} \qquad \qquad \scalebox{.75}{
    \begin{ytableau}
         1 & 2 \\
         3 & 4 & 5 
    \end{ytableau}
    \qquad 
    \begin{ytableau}
         1 & 3 \\
         2 & 4 & 5 
    \end{ytableau}}$$
\end{ex}

The extended Schur and row-strict extended Schur functions are related by $\psi$, most easily seen via their $F$-expansions.  This relationship follows from the fact that the set of standard tableaux is the same but the definitions of descent sets are in a sense complementary and the map $\psi$ extends the complement map on compositions.

\begin{prop}\label{psi_on_eschur} \cite{row0hecke}
    For a composition $\alpha$,  $$\psi(\shin^*_{\alpha}) = \rshin^*_{\alpha} \text{\quad and \quad} \psi(\shin_{\alpha}) = \rshin_{\alpha}.$$ Moreover,  $\{\rshin^*_{\alpha}\}_{\alpha}$ is a basis of $QSym$ 
 and $\{\rshin_{\alpha}\}_{\alpha}$ is a basis of \textit{NSym}.
\end{prop}

Recall that $\psi$ on $QSym$ restricts to $\omega$ on $Sym$, so it follows that $\rshin^*_{\lambda} = s_{\lambda'}$ for a partition $\lambda$. This is also easily seen from the (cancellation-free) expansions of the row-strict extended Schur functions into the monomial and fundamental bases, which follow from Equation \eqref{rshin*_to_F}.  For a composition $\alpha$, 

$$ \rshin^*_{\alpha} = \sum_{\beta} \mathcal{K}^{\rshin}_{\alpha, \beta} M_{\beta} \text{\quad and \quad} \rshin^*_{\alpha} = \sum_{\beta} \mathcal{L}^{\rshin}_{\alpha, \beta} F_{\beta},$$ where $\mathcal{K}^{\rshin}_{\alpha, \beta}$ is the number of row-strict shin tableaux of shape $\alpha$ and type $\beta$, and $\mathcal{L}^{\rshin}_{\alpha, \beta}$ is the number of standard row-strict shin tableaux of shape $\alpha$ with descent composition $\beta$. Dually, for a composition $\beta$, 

\begin{equation} \label{HR_to_rshin}
    H_{\beta} = \sum_{\alpha} \mathcal{K}^{\rshin}_{\alpha, \beta} \rshin_{\alpha} \text{\qquad and \qquad} R_{\beta} = \sum_{\alpha} \mathcal{L}^{\rshin}_{\alpha, \beta} \rshin_{\alpha}.
\end{equation}

We can apply $\psi$ to various results on the shin and extended Schur bases to find analogous results on the row-strict shin and row-strict extended Schur bases.

\begin{thm}\label{rshin_big_thm} For compositions $\alpha$, $\beta$ and a positive integer $m$,
    \begin{enumerate}
        \item Right Pieri Rule. $$\rshin_{\alpha} E_m = \sum_{ \alpha \subset^{\shin}_m \beta} \rshin_{\beta}.$$
        \item Right Ribbon Multiplication. $$\rshin_{\alpha} R_{\beta} = \sum_{\gamma \models |\alpha| + |\beta|} \sum_{U} \rshin_{\gamma},$$ where the sum runs over all skew standard row-strict shin-tableaux $U$ of shape $\gamma/\alpha$ with $co_{\rshin}(U) = \beta$.
        \item $\displaystyle E_{\beta} = \sum_{\alpha} \mathcal{K}_{\alpha,\beta} \rshin_{\alpha} \text{\qquad and \qquad} R_{\beta} = \sum_{\alpha} \mathcal{L}_{\alpha, \beta^c} \rshin_{\alpha}.$
        \item $\chi(\rshin_{\lambda})= s_{\lambda'}  \text{ for a partition $\lambda$, and } \chi(\rshin_{\alpha})=0  \text{ when $\alpha$ is not a partition.}$\vspace{4mm}
        \item Let $\gamma$ be a composition such that $\gamma_i < \gamma_{i+1}$ for all $1 \leq i < \ell(\gamma)$. Then, 
        $$\rshin_{\gamma} = \sum_{\sigma \in S_{\ell(\gamma)}} (-1)^{\sigma} E_{\gamma_{\sigma(1)}}E_{\gamma_{\sigma(2)}} \cdots E_{\gamma_{\sigma(\ell(\gamma))}},$$ where the sum runs over $\sigma \in S_{\ell(\gamma)}$ such that $\sigma(i) \geq i-1$ for all $i \in [\ell(\gamma)]$.\vspace{2mm}
        \item $ \displaystyle E_{\beta} = \sum_{\alpha} \mathcal{K}^{\rshin}_{\alpha, \beta} \shin_{\alpha} \text{\qquad and \qquad} R_{\beta} = \sum_{\alpha} \mathcal{L}^{\rshin}_{\alpha, \beta^c} \shin_{\alpha}.$
    \end{enumerate}
\end{thm}

\begin{proof} For a composition $\alpha$, note that $\psi(H_{\alpha}) = E_{\alpha}$ \cite{Luoto2013}.  
 \hfill
    \begin{enumerate}
        \item Apply $\psi$ to Definition \ref{shin_r_pieri}.
        \item Applying $\psi$ to the LHS of Theorem \ref{shin_right_R_mult} yields $\rshin_{\alpha}R_{\beta^c}$ while the RHS yields $ \displaystyle \sum_{\gamma \models |\alpha|+|\beta|} \sum_U \rshin_{\gamma}$ where the sum runs over all skew standard shin-tableaux $U$ of shape $\gamma / \alpha$ and descent composition $\beta$. By switching $\beta$ with $\beta^c$ everywhere, we can rewrite this equality as $ \displaystyle \rshin_{\alpha}R_{\beta} = \sum_{\gamma \models |\alpha|+|\beta|} \sum_{U} \rshin_{\gamma}$ where the sum runs over all skew standard shin-tableaux $U$ of shape $\gamma/\alpha$ with descent composition $\beta^c$.  Recall that these are equivalent to the skew standard row-strict shin-tableaux of shape $\gamma/\alpha$ with descent composition $\beta$. Our claim follows.
        \item Apply $\psi$ to Equation \eqref{HR_to_shin}.
        \item Based on the restriction of $\omega$ to $Sym$, we have $\psi(\shin^*_{\lambda}) = \omega(s_{\lambda}) = s_{\lambda'}$.  Because $\chi$ is dual to the inclusion map from $Sym$ to $QSym$ \cite{Camp}, we have $\chi(\rshin_{\alpha}) = \sum_{\lambda} (\text{coefficient of $\rshin^*_{\alpha}$ in $s_{\lambda}$}) s_{\lambda}$.  Our claim follows.
        \item Apply $\psi$ to Theorem \ref{JT_shin}.
        \item Apply $\psi$ to Equation \eqref{HR_to_rshin}. \qedhere
    \end{enumerate}
\end{proof}

We also develop row-strict versions of our results from Section \ref{section_skew}.

\begin{defn}
    For compositions $\alpha$ and $\beta$ with $\beta \subseteq \alpha$, the \emph{skew row strict extended Schur functions} are defined by $$\rshin^*_{\alpha/\beta} = \rshin^{\perp}_{\beta}(\rshin^*_{\alpha}).$$
\end{defn}

By Equation \eqref{NSym_perp}, $\rshin^*_{\alpha/\beta}$ expands into various bases as follows. For compositions $\alpha$ and $\beta$ with $\beta \subseteq \alpha$, 
\begin{equation}\label{rshin_skew_coeffs}
    \rshin^*_{\alpha/\beta}= \sum_{\gamma} \langle \rshin_{\beta}H_{\gamma}, \rshin^*_{\alpha}  \rangle M_{\gamma} = \sum_{\gamma} \langle \rshin_{\beta} R_{\gamma}, \rshin^*_{\alpha} \rangle F_{\gamma} = \sum_{\gamma} \langle \rshin_{\beta}\rshin_{\gamma}, \rshin^*_{\alpha} \rangle \rshin^*_{\gamma}.
\end{equation}
Like before, the coefficients $\langle \rshin_{\beta} \rshin_{\gamma}, \rshin^*_{\alpha}  \rangle$ are the coefficients that appear when multiplying row-strict shin functions, $$\rshin_{\beta} \rshin_{\gamma} = \sum_{\alpha} \langle \rshin_{\beta} \rshin_{\gamma}, \rshin^*_{\alpha}  \rangle \rshin_{\alpha}.$$

Because $\psi$ is an automorphism in both $QSym$ and \textit{NSym}, it maps the skew row-strict extended Schur to the extended Schur functions.

\begin{prop}
    For compositions $\alpha$ and $\beta$ with $\beta \subseteq \alpha$, $$\psi(\shin^*_{\alpha/\beta}) = \rshin^*_{\alpha/\beta}.$$ 
\end{prop}

\begin{proof}
First, observe that $$\langle \shin_{\beta} \shin_{\gamma}, \shin^*_{\alpha} \rangle = \langle \psi(\shin_{\beta} \shin_{\gamma}), \psi(\shin^*_{\alpha}) \rangle = \langle \rshin_{\beta} \rshin_{\gamma}, \rshin^*_{\alpha} \rangle,$$ because $\psi$ is invariant under duality and an automorphism in \textit{NSym}. Then, 
$$\psi(\shin^*_{\alpha/\beta})= \sum_{\gamma} \langle \shin_{\beta} \shin_{\gamma}, \shin^*_{\alpha} \rangle \psi(\shin^*_{\gamma}) = \sum_{\gamma} \langle \rshin_{\beta} \rshin_{\gamma}, \rshin^*_{\alpha} \rangle \rshin^*_{\gamma}.$$ By Equation \eqref{rshin_skew_coeffs}, this implies $\psi(\shin^*_{\alpha/\beta}) = \rshin^*_{\alpha/\beta}$. 
\end{proof}

The first line of the proof above yields the following.

\begin{cor}
For compositions $\alpha, \beta$ and $\gamma$, $$\rshin_{\beta} \rshin_{\gamma} = \sum_{\alpha} \mathcal{C}^{\alpha}_{\beta, \gamma} \rshin_{\alpha}.$$
\end{cor}

By Proposition \ref{shin_struct_coeff}, many of the coefficients of the skew row-strict extended Schur functions expanded into the row-strict extended Schur basis are zero while others equal certain Littlewood-Richardson coefficients. Thus, the same is true for the structure coefficients of the row-strict shin functions. 

Like the skew extended Schur functions, the skew row-strict extended Schur functions are defined combinatorially via a class of skew tableaux. 

\begin{defn} For compositions $\alpha = (\alpha_1, \ldots, \alpha_k)$  and $\beta = (\beta_1, \ldots, b_{\ell})$ such that $\beta \subseteq \alpha$, a \emph{skew row-strict shin-tableau} of skew shape $\alpha/\beta$ is a diagram $\alpha/\beta$ filled with integers such that each row is strictly increasing left to right, each column is weakly increasing top to bottom, and if $\alpha_i> \beta_i$ for any  $1 \leq i \leq k$, then $\beta_j < \beta_i$ for all $j > i$.
A skew row-strict shin-tableau is \emph{standard} if it contains the numbers $1$ through $|\alpha| - |\beta|$ each exactly once. 
\end{defn}

\begin{ex}
     The three leftmost diagrams are skew row-strict shin-tableaux while the rightmost diagram is not.

    $$ \text{Skew row-strict shin:\quad}
    \scalebox{.85}
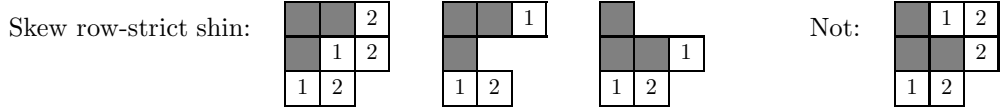
{
    \begin{ytableau}
        *(gray) & *(gray) & 2\\
        *(gray) &  1 & 2\\
        1 & 2
    \end{ytableau}
    \qquad 
    \begin{ytableau}
         *(gray) & *(gray) & 1\\
        *(gray)   \\
        1 & 2
    \end{ytableau}
    \qquad
    \begin{ytableau}
        *(gray)\\ 
        *(gray) & *(gray) & 1 \\
        1 & 2
    \end{ytableau}}
     \qquad
     \qquad
     \text{Not:\quad }
     \scalebox{.85}{
     \begin{ytableau}
         *(gray) & 1 & 2 \\
        *(gray) & *(gray) & 2 \\
        1 & 2
     \end{ytableau}}
    $$
\end{ex}

\begin{prop}\label{rskew_rshin_tab}
    For compositions $\alpha$ and $\beta$ where $\beta \subseteq \alpha$, $$\rshin^*_{\alpha/\beta} = \sum_{T} x^T,$$ where the sum runs over all skew row-strict shin-tableaux $T$ of shape $\alpha/\beta$.
\end{prop}

\begin{proof}
    By Theorem \ref{rshin_big_thm} (2), the coefficient $\langle \rshin_{\beta}R_{\gamma}, \rshin^*_{\alpha} \rangle$ is the number of skew row-strict shin-tableaux of shape $\alpha/\beta$  with descent composition $\gamma$, which we denote $\mathcal{L}^{\rshin}_{\alpha/\beta, \gamma}$. Then by Equation \eqref{rshin_skew_coeffs}, we have $$\rshin^*_{\alpha/\beta} = \sum_{\gamma} \mathcal{L}^{\rshin}_{\alpha/\beta, \gamma} F_{\gamma}.$$ From here it is simple to expand $F_{\gamma}$ into a sum over tableaux and obtain our claim.
\end{proof}

\begin{ex}
    The skew row-strict extended Schur function $\rshin^*_{(1,3,2)/(1,2)}$ expands in terms of the fundamental basis as $$\rshin^*_{(1,3,2)/(1,2)} = F_{2,1} + F_{1,2} + F_{1,1,1}. \qquad \qquad \scalebox{.85}{\begin{ytableau}
        *(gray)\\ 
        *(gray) & *(gray) & 1 \\
        2 & 3
    \end{ytableau} 
    \qquad
     \begin{ytableau}
        *(gray)\\ 
        *(gray) & *(gray) & 2 \\
        1 & 3
    \end{ytableau}
    \qquad
     \begin{ytableau}
        *(gray)\\ 
        *(gray) & *(gray) & 3 \\
        1 & 2
    \end{ytableau}} \vspace{2mm}$$
    This further expands in terms of skew row-strict shin tableaux as

$$\rshin^*_{(1,3,2)/(1,2)} = x_1^2 x_2 + x_1 x_2 x_3 + x_1x_2^2 + x_1x_2x_3 + x_1 x_2 x_3 \cdots \vspace{2mm}$$
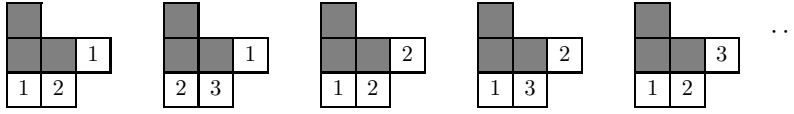
    $$\scalebox{.85}{ \begin{ytableau}
        *(gray)\\ 
        *(gray) & *(gray) & 1 \\
        1 & 2
    \end{ytableau}
    \quad \quad
     \begin{ytableau}
        *(gray)\\ 
        *(gray) & *(gray) & 1 \\
        2 & 3
    \end{ytableau}
    \quad \quad
     \begin{ytableau}
        *(gray)\\ 
        *(gray) & *(gray) & 2 \\
        1 & 2
    \end{ytableau}
    \quad \quad
     \begin{ytableau}
        *(gray)\\ 
        *(gray) & *(gray) & 2 \\
        1 & 3
    \end{ytableau}
    \quad \quad
     \begin{ytableau}
        *(gray)\\ 
        *(gray) & *(gray) & 3 \\
        1 & 2
    \end{ytableau}}\quad \cdots \vspace{2mm}$$
\end{ex}

\subsection{Reverse extended Schur and shin functions}

Let $\alpha$ be a composition and $\beta$ a weak composition. A \emph{reverse shin-tableau} (FST) of shape $\alpha$ and type $\beta$ is a diagram $\alpha$ filled with positive integers that weakly decrease along the rows from left to right and strictly increase along the columns from top to bottom, where each positive integer $i$ appears $\beta_i$ times.  A \emph{standard} reverse shin-tableau (SFST) of shape $\alpha \models n$ is one containing the entries $1$ through $n$ each exactly once. 

\begin{ex} A few reverse shin-tableaux of shape $(2,3)$ are 
    $$\scalebox{.85}
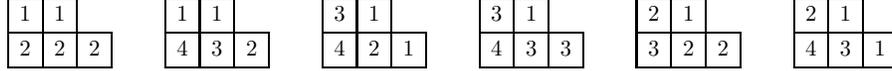
{
    \begin{ytableau}
        1& 1\\
        2&2 &2
    \end{ytableau}
    \qquad
    \begin{ytableau}
        1 & 1\\
        4&3 &2
    \end{ytableau}
    \qquad
    \begin{ytableau}
        3& 1\\
        4&2 &1
    \end{ytableau}
    \qquad
    \begin{ytableau}
        3& 1\\
        4&3 &3
    \end{ytableau}
    \qquad
    \begin{ytableau}
        2& 1\\
        3&2 &2
    \end{ytableau}
    \qquad
    \begin{ytableau}
        2& 1\\
        4&3 &1
    \end{ytableau}}
    $$
\end{ex}

\begin{defn}
    For a composition $\alpha$, the \emph{reverse extended Schur function} is defined as $$\fshin^*_{\alpha} = \sum_T x^T,$$ where the sum runs over all reverse shin-tableaux $T$ of shape $\alpha$.
\end{defn}

 The \emph{descent set} is defined as $Des_{\fshin}(S) = \{i : i+1 \text{ is strictly below $i$ in $S$} \}$ for a standard reverse shin-tableau $S$. Each entry $i$ in $Des_{\fshin}(S)$ is called a \emph{descent} of $S$. The \emph{descent composition} of $S$ is defined $co_{\fshin}(S) = (i_1, i_2 - i_1, \ldots, i_{d}-i_{d-1}, n-i_{d})$ for $Des_{\fshin}(S) = \{i_1, \ldots, i_d\}$. The \emph{reverse shin-reading word} of a reverse shin-tableau $T$, denoted $rw_{\fshin}(T)$ is obtained by reading the rows of $T$ from right to left starting with the bottom row and moving up. To \emph{standardize} a reverse shin-tableau $T$, replace the $1$'s in $T$ with $1,2,\ldots$ in the order they appear in $rw_{\fshin}(T)$, then the $2$'s starting with the next consecutive number, etc.

\begin{prop}\label{fshin*_to_F} For a composition $\alpha$, $$\fshin^*_{\alpha} = \sum_{S} F_{co_{\fshin}(S)},$$ where the sum runs over standard reverse shin-tableaux $U$ of shape $\alpha$.
\end{prop}

\begin{proof}
    We can write $\fshin^*_{\alpha} = \sum_{S} \sum_{std(T)=S} x^T$ where the sums run over standard reverse shin-tableaux $S$ of shape $\alpha$ and reverse shin-tableaux $T$ that standardize to $S$. Now we want to show that, given a SFST $S$, we can write $F_{co_{\fshin}(S)} = \sum_{std(T) =S} x^T$ where the sum runs over FST $T$ that standardize to $S$. First, observe that given $T$ such that $std(T)=S$, if $i \in Des_{\fshin}(S)$ then the box in $T$ corresponding to $i+1$ is strictly below the box corresponding to $i$.  Therefore, by the order of standardization (right to left, bottom to top), the box corresponding to $i+1$ must be filled with a strictly higher number than the box corresponding to $i$.  It follows that $\widetilde{type(T)}$ is a refinement of the descent composition of $S$. In fact, any refinement $B$ of the descent composition of $S$ is a possible type of a tableau $T$ such that $std(T)=S$ because it is associated with a valid filling. Each box of $S$ corresponds to a letter in $B$. If a box of $S$ corresponds to a letter in word $v_i$ of $B$, then fill that same box in $T$ with an $i$.  This is equivalent to reading through the boxes $T$ in the order they appear in $S$, and filling them based on the location of the corresponding letter in $B$. Note that this method creates the unique tableaux $T$ with a specific type that standardizes to $S$ because we have used the only possible order of filling to maintain our desired type and standardization.  Thus, $$F_{co_{\fshin}(S)} = \sum_{\beta \preceq co_{\fshin}(S)} M_{\beta} = \sum_{\beta \preceq co_{\fshin}(S)} \sum_{\widetilde{type(T)} = \beta} x^T = \sum_{std(T)=S} x^T.$$ Therefore, we have $\fshin^*_{\alpha} = \sum_S F_{co_{\fshin}(S)}$ where the sum runs over SFSTs of shape $\alpha$. 
\end{proof}

\begin{ex}\label{fshin_F_ex} The $F$-expansion of the reverse extended Schur function $\fshin^*_{(3,2)}$ and the standard reverse shin-tableaux of shape $(3,2)$ are:
$$\fshin^*_{(3,2)} = F_{(3,2)} + F_{(2,2,1)} \qquad \qquad \scalebox{.75}{\begin{ytableau}
  3 & 2 & 1 \\
  5 & 4
\end{ytableau}
\qquad
\begin{ytableau}
   4 & 2 & 1 \\
   5 & 3
\end{ytableau}}$$
\end{ex}

Let $\mathcal{K}^{\fshin}_{\alpha, \beta}$ be the number of FST of shape $\alpha$ and type $\beta$, and let $\mathcal{L}^{\fshin}_{\alpha, \beta}$ be the number of SFST with shape $\alpha$ and descent composition $\beta$. Using Proposition \ref{fshin*_to_F}, it is straightforward to show that the reverse extended Schur functions have the following positive (cancellation-free) expansions into the monomial and fundamental bases. For a composition $\alpha$, \begin{equation}\label{fshin*_expansions}
\fshin^*_{\alpha} = \sum_{\beta} \mathcal{K}^{\fshin}_{\alpha, \beta} M_{\beta} \text{\quad and \quad}  \fshin^*_{\alpha}= \sum_{\beta} \mathcal{L}^{\fshin}_{\alpha, \beta} F_{\beta}.\end{equation}

To understand the relationship between standard shin-tableaux and standard reverse shin-tableaux, we establish a bijection $flip: \{ \text{standard shin-tableaux} \} \rightarrow \{ \text{standard reverse shin-tableaux} \}$.  Define $flip(S)$ to be the tableau $U$ obtained by flipping $S$ horizontally (in other words, reversing the order of the rows of $S$) and then replacing each entry $i$ with $n+1-i$. The map $flip$ is an involution between the set of standard shin-tableaux and the set of standard reverse shin-tableaux.

\begin{ex} The map $flip$ works on the following tableau as follows:
$$ flip \left(\ \scalebox{.85}{\ \begin{ytableau} 1 & 3 & 4 \\
2 & 5 \end{ytableau}}\ \right) \quad = \quad \scalebox{.85}{ \begin{ytableau} 4 & 1 \\
5 & 3 & 2 \end{ytableau}}$$
\end{ex}

By construction, the descent composition of a standard shin-tableau $U$ is the reverse of the descent composition of the standard reverse shin-tableau given by $flip(U)$. Using this fact, we show that the reverse extended Schur functions are the image of the extended Schur functions under $\rho$.

\begin{thm}\label{rho_on_shin} For a composition $\alpha$, $$\rho(\shin^*_{\alpha}) = \fshin^*_{\alpha^r} \text{\quad and \quad} \omega(\rshin^*_{\alpha}) = \fshin^*_{\alpha^r}.$$ Moreover, $\{\fshin^*_{\alpha}\}_{\alpha}$ is a basis of $QSym$. 
\end{thm}

\begin{proof}
    Let $\alpha \models n$. First, observe that given a standard shin-tableau $U$ of shape $\alpha$ and a standard reverse shin-tableau $S$ of shape $\alpha^r$ with $flip(U) = S$, we have $(co_{\shin}(U))^r = co_{\fshin}(S)$.  Therefore,
    \begin{align*}
    \rho(\shin^*_{\alpha}) &= \rho(\sum_U F_{co_{\shin}(U)}) = \sum_U \rho(F_{co_{\shin}(U)})\\
    &= \sum_U F_{co_{\shin}(U)^r} = \sum_{flip(U)} F_{co_{\fshin}(flip(U))} = \sum_S F_{co_{\fshin}(S)},  
    \end{align*}
    where the sums run over SST $U$ of shape $\alpha$ and standard reverse shin-tableaux $S$ of shape $\alpha^r$. The fact that the reverse extended Schur functions are a basis follows from $\rho$ being an automorphism in $QSym$. 
\end{proof}

\begin{prop}
    The reverse extended Schur basis is not equivalent to the extended Schur basis or the row-strict extended Schur basis.
\end{prop}

\begin{proof}
    From Example \ref{fshin_F_ex} above, we can see that there is no $\beta$ such that $\shin^*_{\beta} = \fshin^*_{(3,2)}$.  The only standard shin-tableaux with descent composition $(3,2)$ are tableaux of shape $(3,2)$ or shape $(4,1)$ meaning $\shin^*_{(3,2)}$ and $\shin^*_{(4,1)}$ are the only extended Schur functions in which $F_{(3,2)}$ appears but neither of them equal $\fshin^*_{(3,2)}$. 
    
    Now consider $\fshin^*_{(1,2,1)} = F_{(1,2,1)} + F_{(2,1,1)}$. The only $\beta$ for which there exists a standard row-strict shin-tableaux of shape $\beta$ with descent compositions $(1,2,1)$ and $(2,1,1)$ is $\beta = (3,1)$.  However, $\rshin^*_{(3,1)} = F_{(3,1)} + F_{(1,2,1)} + F_{(1,1,2)}$.  Thus, there is no $\beta$ such that $\fshin^*_{(1,2,2)} = \rshin^*_{\beta}$.
\end{proof}

Next, we consider the basis of \textit{NSym} that is dual to the reverse extended Schur functions and its relationship with the shin functions.

\begin{defn} Define the \emph{reverse shin basis} $\{ \fshin_{\alpha}\}_{\alpha}$ as the unique basis of \textit{NSym} that is dual to the reverse extended Schur basis. Equivalently, $\langle \fshin_{\alpha}, \fshin^*_{\beta} \rangle = \delta_{\alpha, \beta}$ for all compositions $\alpha$ and $\beta$. 
\end{defn}

The expansions of the complete homogeneous functions and the monomial functions of \textit{NSym} into the reverse shin basis follow via duality from Equation \eqref{fshin*_expansions}.  For a composition $\beta$,
\begin{equation}\label{HR_to_fshin}
H_{\beta} = \sum_{\alpha} \mathcal{K}^{\fshin}_{\alpha, \beta} \fshin_{\alpha} \text{\quad and \quad} R_{\beta} = \sum_{\alpha} \mathcal{L}^{\fshin}_{\alpha, \beta} \fshin_{\alpha}.
\end{equation}

Like in the dual case, the reverse shin functions are related to the shin functions via the involution $\rho$.

\begin{prop}
    For a composition $\alpha$, we have $$\fshin_{\alpha} = \rho(\shin_{\alpha^r}) \quad \fshin_{\alpha} = \omega(\rshin_{\alpha^r}).$$
\end{prop}

\begin{proof}
The map $\rho$ is invariant under duality.  By definition, $ \langle R_{\alpha}, F_{\beta} \rangle  = \delta_{\alpha, \beta} = \langle R_{\alpha^r}, F_{\beta^r} \rangle = \langle \rho(R_{\alpha}), \rho(F_{\beta}) \rangle, $ thus for $G \in QSym$ and $H \in$ \textit{NSym}, $\langle H, G \rangle = \langle \rho(H), \rho(G) \rangle$ due to the bilinearity of the inner product. Then the first part of our claim follows from Theorem \ref{rho_on_shin}. The relationship with the row-strict shin functions follows from $\omega = \rho \circ \psi$ and Proposition \ref{psi_on_eschur}.
\end{proof}

By applying $\rho$, we can translate many of the results on the shin functions to the reverse shin functions.

\begin{thm} For compositions $\alpha, \beta$, a partition $\lambda$, and a positive integer $m$,
    \begin{enumerate}
        \item Left Pieri Rule. $$H_m \fshin_{\alpha} = \sum_{\alpha^r \subset^{\shin}_{m} \beta^r} \fshin_{\beta}.$$
        \item $ \displaystyle H_{\beta} = \sum_{\alpha} \mathcal{K}_{\alpha^r, \beta^r} \fshin_{\alpha} \text{\qquad and \qquad} R_{\beta} = \sum_{\alpha} \mathcal{L}_{\alpha^r, \beta^r} \fshin_{\alpha}.$
        \item $\fshin^*_{\lambda^r} = s_{\lambda}. \text{\quad Also, } \chi(\fshin_{\lambda^r}) = s_{\lambda} \text{  and  } \chi(\fshin_{\alpha})=0 \text{ when $\alpha$ is not weakly increasing.}$\vspace{4mm} 
        \item Let $\gamma$ be a composition such that $\gamma_i > \gamma_{i+1}$ for all $1 \leq i \leq \ell(\gamma)$. Then,
        $$\fshin_{\gamma} = \sum_{\sigma \in S_{\ell(\gamma)}} (-1)^{\sigma} H_{\gamma_{\sigma(1)}} \cdots H_{\gamma_{\sigma(\ell(\gamma))}},$$ where the sum runs over $\sigma \in S_{\ell(\gamma)}$ such that $\sigma(i) \geq i-1$ for all $i \in [\ell(\gamma)]$. \vspace{2mm}
        \item $\displaystyle H_{\beta} = \sum_{\alpha} \mathcal{K}^{\fshin}_{\alpha^r, \beta^r} \shin_{\alpha} \text{\quad and \quad} R_{\beta} = \sum_{\alpha} \mathcal{L}^{\fshin}_{\alpha^r, \beta^r} \shin_{\alpha}.$
    \end{enumerate}
\end{thm}

\begin{proof} For a composition $\alpha$, note that $\rho(H_{\alpha}) = H_{\alpha^r}.$ 
 \hfill
   \begin{enumerate}
        \item Apply $\rho$ to Definition \ref{shin_r_pieri}.
        \item Apply $\rho$ to Equation \eqref{HR_to_shin}.
        \item Since $\rho$ restricts to the identity map on the Schur functions, we have $\fshin^*_{\lambda^r} = \rho(\shin^*_{\lambda}) = \rho(s_{\lambda}) = s_{\lambda}$. Because $\chi$ is dual to the inclusion map from $Sym$ to $QSym$ \cite{Camp}, we have \newline $\chi(\fshin_{\alpha}) = \sum_{\lambda} (\text{ coefficient of $\fshin^*_{\alpha}$ in $s_{\lambda}$}) s_{\lambda}$.  Our claim follows.
        \item Apply $\rho$ to  Theorem \ref{JT_shin}.
        \item Apply $\rho$ to Equation \eqref{HR_to_fshin}. \qedhere
    \end{enumerate}
\end{proof}

Next, we define skew reverse extended Schur functions algebraically, and then we define \otherskew\ reverse extended Schur functions algebraically and in terms of tableaux. 

\begin{defn}
    For compositions $\beta \subseteq \alpha$, the \emph{skew reverse extended Schur functions} are defined by $$\fshin^*_{\alpha/\beta} = \fshin^{\perp}_{\beta}(\fshin^*_{\alpha}).$$ 
\end{defn}

As before, $\fshin^*_{\alpha/\beta}$ expands into various bases according to Equation \eqref{NSym_perp}. For compositions $\beta \subseteq \alpha$, 
\begin{equation}\label{fshin_skew_exp}
     \fshin^*_{\alpha/\beta}= \sum_{\gamma} \langle \fshin_{\beta}H_{\gamma}, \fshin^*_{\alpha}  \rangle M_{\gamma} = \sum_{\gamma} \langle \fshin_{\beta}R_{\gamma}, \fshin^*_{\alpha} \rangle F_{\gamma} = \sum_{\gamma} \langle \fshin_{\beta}\fshin_{\gamma}, \fshin^*_{\alpha} \rangle \fshin^*_{\gamma}.
\end{equation}

Like before, the coefficients $\langle \shin_{\beta} \shin_{\gamma}, \shin^*_{\alpha}  \rangle$ also appear when multiplying row-strict shin functions, $$\fshin_{\beta} \fshin_{\gamma} = \sum_{\alpha} \langle \fshin_{\beta} \fshin_{\gamma}, \fshin^*_{\alpha}  \rangle \fshin_{\alpha}.$$

Because $\rho$ is an anti-automorphism in \textit{NSym}, it does not map the skew extended Schur functions to the skew reverse extended Schur functions. Instead, the image is the \otherskew\ reverse extended Schur functions. 

\begin{defn}\label{skewii_reverse}
    For compositions $\alpha$ and $\beta$ where $\beta^r \subseteq \alpha^r$, the \emph{\otherskew\  reverse extended Schur functions} are defined by $$\fshin^*_{\alpha\rskew \beta} = \fshin^{\rperp}_{\beta}(\fshin^*_{\alpha}).$$
\end{defn}

Using Definition \ref{NSym_rperp}, $\fshin^*_{\alpha \rskew \beta}$ expands into various bases as follows. For compositions $\alpha, \beta$ with $\beta^r \subseteq \alpha^r$, 
\begin{equation}\label{rskew_fshin_exp} \fshin^*_{\alpha \rskew \beta} = \sum_{\gamma} \langle H_{\gamma}\fshin_{\beta}, \fshin^*_{\alpha} \rangle M_{\gamma} = \sum_{\gamma} \langle R_{\gamma}\fshin_{\beta}, \fshin^*_{\alpha} \rangle F_{\gamma} = \sum_{\gamma} \langle \fshin_{\gamma} \fshin_{\beta}, \fshin^*_{\alpha} \rangle \fshin^*_{\gamma}.   
\end{equation}

\begin{thm}
    For compositions $\alpha$ and $\beta$ where $\beta \subseteq \alpha$, $$\rho(\shin^*_{\alpha/\beta}) = \fshin^*_{\alpha^r \rskew \beta^r}.$$
\end{thm}

\begin{proof}
First, observe that $$\langle \rho(\shin_{\beta} \shin_{\gamma}), \rho(\shin^*_{\alpha}) \rangle = \langle \rho(\shin_{\gamma}) \rho(\shin_{\beta}), \rho(\shin^*_{\alpha}) \rangle = \langle \fshin_{\gamma^r} \fshin_{\beta^r}, \fshin_{\alpha^r} \rangle,$$ because $\rho$ is invariant under duality and an anti-automorphism in \textit{NSym}. Then, 
$$\rho(\shin^*_{\alpha/\beta}) = \sum_{\gamma} \langle \shin_{\beta}\shin_{\gamma}, \shin^*_{\alpha} \rangle 
 \rho(\shin^*_{\gamma}) = \sum_{\gamma} \langle \shin_{\beta}\shin_{\gamma}, \shin^*_{\alpha} \rangle 
 \fshin^*_{\gamma^r} = \sum_{\gamma} \langle \fshin_{\gamma^r} \fshin_{\beta^r}, \fshin_{\alpha^r} \rangle \fshin^*_{\gamma^r} = \fshin^*_{\alpha^r \rskew \beta^r},$$ by Equation \eqref{rskew_fshin_exp}.
\end{proof}

The first line of the proof above yields the following.

\begin{cor}
    For compositions $\beta$ and $\gamma$, 
    $$\fshin_{\gamma}\fshin_{\beta} = \sum_{\alpha} \mathcal{C}^{\alpha^r}_{\beta^r, \gamma^r} \fshin_{\alpha}.$$
   
\end{cor}

The skew-II reverse extended Schur functions are defined combinatorially via a special class of tableaux with \otherskew\ shapes. 

\begin{defn}
    For compositions $\beta = (\beta_1, \ldots, b_{\ell}) $ and $ \alpha = (\alpha_1, \ldots, \alpha_k)$ such that $\beta^r \subseteq \alpha^r$, a \emph{\otherskew\ reverse shin-tableau} of shape $\alpha \rskew \beta$ is a \otherskew\ diagram $\alpha \rskew \beta$ filled with integers such that each row is weakly decreasing left to right, each column is strictly increasing top to bottom, and if $\alpha_{k-i}> \beta_{\ell-i}$ for some $i \geq 0$ then there is no $j>i$ such that $\beta_{\ell-j}> \beta_{\ell - i}$. A \otherskew\  reverse shin-tableau is \emph{standard} if it contains the numbers $1$ through $|\alpha| - |\beta|$ each exactly once. 
\end{defn}

Intuitively, the third condition states that there should never be any boxes directly below a box that has been skewed-out.

\begin{ex}
        The three leftmost diagrams below are examples of  \otherskew\ reverse shin-tableaux while the rightmost diagram is not.

    $$ \text{\otherskew\ reverse shin:\quad}
    \scalebox{.85}
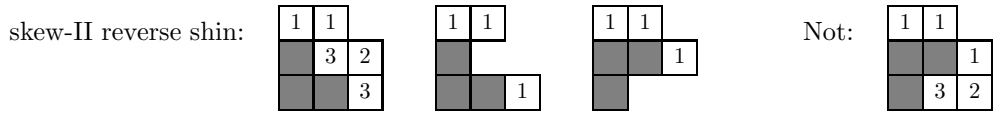
{
    \begin{ytableau}
          1 & 1 \\
        *(gray) & 3 & 2\\
        *(gray) &  *(gray)  & 3
    \end{ytableau}
    \qquad 
    \begin{ytableau}
        1 & 1 \\
        *(gray)\\ 
        *(gray) & *(gray) & 1 
    \end{ytableau}
    \qquad
    \begin{ytableau}
        1 & 1\\ 
         *(gray) & *(gray) & 1\\
        *(gray)   
    \end{ytableau}}
     \qquad
     \qquad
     \text{Not:\quad }
     \scalebox{.85}{
     \begin{ytableau}
        1 & 1\\
         *(gray) & *(gray) & 1 \\
        *(gray) & 3 & 2 
     \end{ytableau}}\vspace{2mm}
    $$ The right-most diagram is not a skew-II reverse shin tableau because $\alpha_3 > \beta_2 $ but $\beta_1> \beta_2$. In other words, there is a skewed-out box directly above the normal box containing the $3$, which violates our conditions.
\end{ex}

We can extend the bijection $flip$ to be $$flip: \{\text{skew standard shin-tableaux of shape } \alpha/\beta\} \rightarrow \{ \text{\otherskew\ SFST of shape } \alpha^r \rskew \beta^r\}$$ where, given a skew standard shin-tableau $T$, the \otherskew\ standard reverse shin-tableau $flip(T)$ is obtained by flipping $T$ horizontally (reversing the order of the rows) and replacing each entry $i$ with $|T|+1-i$. Again, by definition of the map, the descent composition of $T$ is the reverse of the descent composition of $flip(T)$. This is because there is a descent at $|T|+1-i$ in $flip(T)$ whenever there is a descent $i$ in $T$. This allows us to connect the skew extended Schur functions with the skew-II reverse extended Schur functions.

\begin{prop}
       For compositions $\alpha$ and $\beta$ where $\beta^r \subseteq \alpha^r$, $$\fshin^*_{\alpha \rskew \beta} = \sum_{T} x^T,$$ where the sum runs over \otherskew\ reverse shin-tableaux $T$ of shape $\alpha \rskew \beta$.
\end{prop}

\begin{proof}
 The skew shin functions can be expressed as $$\shin^*_{\alpha/\beta} = \sum_{S} F_{co_{\shin}(S)},$$ using the same logic as the usual shin functions (also similar to Proposition \ref{fshin*_to_F}). Thus, $$\fshin^*_{\alpha \rskew \beta} = \omega(\shin^*_{\alpha^r/\beta^r}) = \sum_{S}F_{co_{\shin}(S)^r} = \sum_{S} F_{co_{\fshin}(flip(S)} = \sum_{U} F_{co_{\fshin}(U)},$$ where the sums run over skew standard  shin-tableaux $S$ of shape $\alpha^r /\beta^r$ and \otherskew\ standard reverse shin-tableaux $U$ of shape $\alpha \rskew \beta$. From here it is simple to expand the summation into our claim.
\end{proof}

With the notion of \otherskew\ reverse shin-tableaux we now define a reverse analogue to Theorem \ref{shin_right_R_mult}.

\begin{prop}\label{fshin_R_mult}
Left Ribbon Multiplication. $$R_{\beta}\fshin_{\alpha} = \sum_{\gamma \models |\alpha| + |\beta|} \sum_{S} \fshin_{\gamma},$$ where the sum runs over all \otherskew\ standard reverse shin-tableaux $S$ of shape $\gamma/\alpha$ with $co_{\fshin}(U) = \beta$. 
\end{prop}

\begin{proof}
For compositions $\alpha$ and $\beta$, Theorem \ref{shin_right_R_mult} yields $$\shin_{\alpha^r} R_{\beta^r} = \sum_{\gamma \models |\alpha| + |\beta|} \sum_{U} \shin_{\gamma^r},$$ where the sum runs over skew standard shin-tableaux $U$ of shape $\gamma^r/\alpha^r$ with $co_{\shin}(U) = \beta^r$.  Then applying $\rho$ gives $$R_{\beta}\fshin_{\alpha} = \sum_{\gamma \models |\alpha| + |\beta|} \sum_{U} \fshin_{\gamma},$$ where the sum runs over skew standard shin-tableaux $U$ of shape $\gamma^r/\alpha^r$ with $co_{\shin}(U) = \beta^r$. Using the $flip$ bijection, we associate each tableau $U$ with $flip(U)$ which is a \otherskew\ standard reverse shin-tableau of shape $\alpha \rskew \beta$ and descent composition $\beta$. This allows us to rewrite our sum as it is stated in the claim.
\end{proof}

\subsection{Row-strict reverse extended Schur and shin functions}

Let $\alpha$ be a composition and $\beta$ be a weak composition. A \emph{row-strict reverse shin-tableau} (BST) of shape $\alpha$ and type $\beta$ is a filling of the diagram of $\alpha$ with positive integers such that the entries in each row are strictly decreasing from left to right and the entries in each column are weakly increasing from top to bottom where each integer $i$ appears $\beta_i$ times. These are essentially a row-strict version of the reverse shin-tableaux.   A row-strict reverse shin-tableau of shape $\alpha \models n$ is \emph{standard} (SBST) if it includes the entries $1$ through $n$ each exactly once.

\begin{defn}
    For a composition $\alpha$, the \emph{row-strict reverse extended Schur function} is defined as $$\bshin^*_{\alpha} = \sum_T x^T,$$ where the sum runs over all row-strict reverse shin-tableaux $T$ of shape $\alpha$.
\end{defn}

\begin{ex} The function $\bshin^*_{(2,3)}$ and a few row-strict reverse shin-tableaux of shape $(2,3)$ are
    $$ \scalebox{.85}{
    \begin{ytableau}
       2 & 1 \\
       3 & 2 & 1
    \end{ytableau}
    \qquad
    \begin{ytableau}
       3 & 2 \\
       3 & 2 & 1
    \end{ytableau}
    \qquad
    \begin{ytableau}
        4 & 2 \\
        3 & 2 & 1
    \end{ytableau}
    \qquad 
    \begin{ytableau}
        4 & 1 \\
        3 & 2 & 1
    \end{ytableau}
    \qquad
    \begin{ytableau}
        4 & 2 \\
        4 & 2 & 1
    \end{ytableau}
     \qquad
    \begin{ytableau}
        4 & 2 \\
        4 & 3 & 1
    \end{ytableau}}
    $$

  $$\bshin^*_{(2,3)} = x_1^2x_2^2x_3 + x_1x_2^2x_3^2 + x_1x_2^2x_3x_4 + x_1^2x_2x_3x_4 + x_1x_2^2x_4^2 + x_1x_2x_3x_4^2 + \cdots$$
\end{ex}

Let $S$ be a SBST. The \emph{descent set} is defined to be $Des_{\bshin}(S) = \{ i : \text{ $i+1$ is weakly above $i$ in $S$}\}$. Each entry $i$ in $Des_{\bshin}(S)$ is called a \emph{descent} of $S$. The \emph{descent composition} of $S$ is defined to be  $co_{\bshin}(S) = (i_1, i_2 - i_1, \ldots, i_{d}-i_{d-1}, n-i_{d})$ for $Des_{\bshin}(S) = \{i_1, \ldots, i_d\}$. Note that the set of standard row-strict reverse shin-tableaux is exactly the same as the set of standard reverse shin-tableaux.

The \emph{row-strict reverse shin reading word} of a shin-tableau $T$, denoted $rw_{\bshin}(T)$, is the word obtained by reading the rows of $T$ from right to left starting with the top row and moving down.
We can \emph{standardize} a standard row-strict reverse shin-tableau as follows. Given a BST $T$, its \emph{standardization} is the SBST obtained by replacing the $1$'s in $T$ with $1,2,\ldots$ in the order they appear in $rw_{\bshin}(T)$, then the $2$'s continuing with our consecutive integers from before, then $3$'s, etc.

As in Proposition \ref{fshin*_to_F}, one can show that the flat type of any row-strict reverse shin-tableau $T$ that standardizes to a SBST $S$ is a refinement of $co_{\bshin}(S)$. From there, we can group tableaux together based on their types and standardizations to expand a row-strict reverse extended Schur function into the fundamental basis as follows.

\begin{prop}\label{bshin*_to_F}
For a composition $\alpha$, $$\bshin^*_{\alpha} = \sum_S F_{co_{\bshin}(S)},$$ where the sum runs over standard row-strict reverse shin-tableaux $S$.
\end{prop}

\begin{ex} The $F$-expansion of the row-strict reverse extended Schur function $\bshin^*_{(3,2)}$ and standard row-strict reverse shin-tableaux of shape $(3,2)$ are:
$$\bshin^*_{(3,2)} = F_{(1,1,2,1)} + F_{(1,2,2)} \qquad \qquad \scalebox{.75}{\begin{ytableau}
  3 & 2 & 1 \\
  5 & 4
\end{ytableau}
\qquad
\begin{ytableau}
   4 & 2 & 1 \\
   5 & 3
\end{ytableau}}$$
\end{ex}

Let $\mathcal{K}^{\bshin}_{\alpha, \beta}$ be the number of row-strict reverse shin-tableaux of shape $\alpha$ and type $\beta$, and let $\mathcal{L}^{\bshin}_{\alpha, \beta}$ be the number of standard row-strict reverse shin-tableaux with shape $\alpha$ and descent composition $\beta$. The expansions of the row-strict reverse extended Schur functions into the monomial and fundamental bases follow from Proposition \ref{bshin*_to_F}. For a composition $\alpha$,
\begin{equation}\label{bshin*_expansions}
\bshin^*_{\alpha} = \sum_{\beta} \mathcal{K}^{\bshin}_{\alpha, \beta} M_{\beta}  \text{\quad and \quad } \bshin^*_{\alpha} = \sum_{\beta} \mathcal{L}^{\bshin}_{\alpha, \beta} F_{\beta}.\end{equation}

Studying the descent compositions of standard row-strict reverse tableaux allows us to relate the row-strict reverse extended Schur functions to our other bases of $QSym$.

\begin{thm} \label{back_eschur} For a composition $\alpha$, 
    $$\omega(\shin^*_{\alpha}) = \bshin^*_{\alpha^r} \text{\quad and \quad} \rho(\rshin^*_{\alpha}) = \bshin^*_{\alpha^r} \text{\quad and \quad} \psi(\fshin^*_{\alpha}) = \bshin^*_{\alpha}. $$Moreover, $\{\bshin^*_{\alpha}\}_{\alpha}$ is a basis of $QSym$.
\end{thm}

\begin{proof} 
Recall that $\sum_U F_{co_{\shin}(U)^r} = \sum_S F_{co_{\fshin}(S)}$, and that the set of standard reverse shin-tableaux is equivalent to the set of standard row-strict reverse shin-tableaux (of shape $\alpha$).  By the complementary definitions of reverse descents and row-strict reverse descents, $Des_{\fshin}(S)$ is the set complement of $Des_{\bshin}(S)$ meaning that $co_{\fshin}(S)^c = co_{\bshin}(S)$. Combining these two statements, we have shown that 
$$\omega(\shin^*_{\alpha}) = \sum_{U} \omega(F_{co_{\shin}(U)}) = \sum_{U} F_{co_{\shin}(U)^t} = \sum_U F_{((co_{\shin(U)})^r)^c} = \sum_{S} F_{co_{\bshin}(S)},$$ where the sums run over SST $U$ of shape $\alpha$ and standard row-strict reverse shin-tableaux $S$ of shape $\alpha^r$. The rest follows from Proposition \ref{psi_on_eschur} and Theorem \ref{rho_on_shin}, and the fact that $\psi \circ \rho = \omega$.
\end{proof}

\begin{prop}
    The row-strict reverse extended Schur basis is not equivalent to the extended Schur basis, the row-strict extended Schur basis, or the reverse extended Schur basis.
\end{prop}

\begin{proof}
    $\bshin^*_{2,1} = F_{1,2}$ and it follows from properties of extended Schur functions that $\rshin^*_{\beta} = F_{\beta}$ if and only if $\beta = (m, 1^k)$ for integers $m \geq 1$ and $k \geq 0$. Thus, there is no $\beta$ such that $\rshin^*_{\beta} = \bshin^*_{2,1}$.

    $\bshin^*_{(1,2,1)} = F_{(2,2)}+F_{(1,3)}$ and the only $\beta$ such that there exist standard shin-tableaux with descent composition $(2,2)$ and $(1,3)$ is $\beta = (3,1)$. However, $\shin^*_{(3,1)} = F_{(2,2)} + F_{(1,3)} + F_{(3,1)}$ and so there is no $\beta$ such that $\shin^*_{\beta} = \bshin^*_{(1,2,1)}$. Further, the only $\beta$ such that there exist standard reverse shin-tableaux with descent composition $(2,2)$ and $(1,3)$ is also $\beta = (3,1)$. However, $\fshin^*_{(3,1)} = F_{(2,2)} + F_{(1,3)} + F_{(3,1)} \not= \bshin^*_{(1,2,1)}$. There is no $\beta$ such that $\fshin^*_{\beta}=\bshin^*_{(1,2,1)}$. 
\end{proof}

Next, we consider the basis of \textit{NSym} that is dual to the row-strict reverse extended Schur functions and its relationship with the shin functions.

\begin{defn}
   Define the \emph{row-strict reverse shin basis} $\{ \bshin_{\alpha}\}_{\alpha}$ as the unique basis of \textit{NSym} that is dual to the row-strict reverse extended Schur basis. Equivalently, $\langle \bshin_{\alpha}, \bshin^*_{\beta} \rangle = \delta_{\alpha, \beta}$ for all compositions $\alpha$ and $\beta$. 
\end{defn}

The expansions of the complete homogeneous functions and the ribbon functions of \textit{NSym} into the row-strict reverse shin basis are dual to those in Equation \eqref{bshin*_expansions}.
For a composition $\beta$,
\begin{equation}\label{HR_to_bshin}
    H_{\beta} = \sum_{\alpha} \mathcal{K}^{\bshin}_{\alpha, \beta} \bshin_{\alpha} \text{\quad and \quad} R_{\beta} = \sum_{\alpha} \mathcal{L}^{\bshin}_{\alpha, \beta} \bshin_{\alpha}.
\end{equation}

Like in the dual case, the row-strict reverse shin functions are related to the shin functions via the involution $\omega$.

\begin{prop}\label{back_shin_omega}  For a composition $\alpha$, we have $$\bshin_{\alpha} = \omega(\shin_{\alpha^r})\text{\quad and \quad} \bshin_{\alpha} = \rho(\rshin_{\alpha^r}) \text{\quad and \quad} \bshin_{\alpha} = \psi(\fshin_{\alpha}).$$
\end{prop}

\begin{proof}
Observe that $$\langle R_{\alpha}, F_{\beta} \rangle = \langle R_{\alpha^t}, F_{\beta^t} \rangle = \langle \omega(R_{\alpha}), \omega(F_{\beta})).$$  This property expands to other bases by linearity. For instance,  
$$\langle \shin_{\alpha^r}, \shin^*_{\beta^r} \rangle = \langle \omega(\shin_{\alpha^r}), \omega(\shin^*_{\beta^r}) \rangle = \langle \omega(\shin_{\alpha^r}), \bshin^*_{\beta} \rangle,$$ for any compositions $\alpha$ and $\beta$. It follows by definition then that $\bshin_{\beta} = \omega(\shin_{\beta^r})$. The rest follows similarly from the invariance of $\rho$ and $\psi$ under duality.
\end{proof}

Now, we apply $\omega$ to the various results on the shin and extended Schur bases to find analogous results on the row-strict reverse shin and row-strict reverse extended Schur bases. 

\begin{thm} For compositions $\alpha$, $\beta$, a partition $\lambda$, and a positive integer $m$,
    \begin{enumerate}
        \item Left Pieri Rule. $$E_m \bshin_{\alpha} = \sum_{ \alpha^r \subset^{\shin}_m \beta^r} \bshin_{\beta}.$$
        \item $ \displaystyle E_{\beta} = \sum_{\alpha} \mathcal{K}_{\alpha^r,\beta^r} \bshin_{\alpha} \text{\qquad and \qquad} R_{\beta} = \sum_{\alpha} \mathcal{L}_{\alpha^r, \beta^t} \bshin_{\alpha}.$
        \item $\bshin^*_{\lambda} = s_{\lambda'}. \text{   Also, } \chi(\bshin_{\lambda})= s_{\lambda'} \text{ and } \chi(\bshin_{\alpha})=0  \text{ when $\alpha$ is not a partition.}$\vspace{4mm}
        \item Let $\gamma$ be a composition such that $\gamma_i > \gamma_{i+1}$. Then, 
        $$\bshin_{\gamma} = \sum (-1)^{\sigma} E_{\gamma_{\sigma(1)}}E_{\gamma_{\sigma(2)}} \cdots E_{\gamma_{\sigma(\ell(\gamma))}},$$ where the sum runs over $\sigma \in S_{\ell(\gamma)}$ such that $\sigma(i) \geq i-1$ for all $i \in [\ell(\gamma)]$. \vspace{2mm}
        \item $\displaystyle E_{\beta} = \sum_{\alpha} \mathcal{K}^{\bshin}_{\alpha^r, \beta^r} \shin_{\alpha} \text{\quad and \quad} R_{\beta} = \sum_{\alpha} \mathcal{L}^{\bshin}_{\alpha^r, \beta^t} \shin_{\alpha}.$
    \end{enumerate}
\end{thm}

\begin{proof} For a composition $\alpha$, note that $\omega(H_{\alpha}) = E_{\alpha^r}.$ \hfill
   \begin{enumerate}
        \item Apply $\omega = \rho \circ \psi$ to Definition \ref{shin_r_pieri}.
        \item Apply $\omega = \rho \circ \psi$ to Equation \eqref{HR_to_shin}.
        \item Since $\omega$ restricts to classic $\omega$ map on the Schur functions, we have $\bshin^*_{\lambda^r} = \omega(\shin^*_{\lambda}) = \omega(s_{\lambda}) = s_{\lambda'}$. Because $\chi$ is dual to the inclusion map from $Sym$ to $QSym$ \cite{Camp}, we have \newline $\chi(\bshin_{\alpha}) = \sum_{\lambda} (\text{ coefficient of $\bshin^*_{\alpha}$} \text{ in $s_{\lambda}$}) s_{\lambda}$.  Our claim follows.
        \item Apply $\omega = \rho \circ \psi$ to  Theorem \ref{JT_shin}.
        \item Apply $\omega = \rho \circ \psi$ to Equation \eqref{HR_to_bshin}. \qedhere
    \end{enumerate}
\end{proof}

Another important part of the Hopf algebra structures of $QSym$ and \textit{NSym} are the \emph{antipodes}. The antipode of a Hopf algebra $\mathcal{A}$ is the unique anti-endomorphism that is a two-sided inverse under the convolution product with the identity map on $\mathcal{A}$.  For more details, see \cite{grinberg}.  We write the antipode of \textit{NSym} as $S:$ \textit{NSym} $\rightarrow$ \textit{NSym} and the antipode of $QSym$ as $S^*: QSym \rightarrow QSym$. Formulas for the antipodes of $QSym$ and \textit{NSym} on the fundamental and ribbon bases were given by Malvenuto and Reutenauer in \cite{MalRet} and Benedetti and Sagan in \cite{Ben}, respectively. For any composition $\alpha$, \begin{equation}
     S^*(F_{\alpha}) = (-1)^{|\alpha|}F_{\alpha^t} \text{\quad and \quad} S(R_{\alpha}) = (-1)^{|\alpha|} R_{\alpha^t}.
\end{equation}
Formulas for the antipodes of other Schur-like bases, especially the immaculate basis, have been studied in various papers \cite{Ben, campbell_antipode}.  An open problem is to describe $S(\shin_{\alpha})$ and $S^*(\shin^*_{\alpha})$ in terms of the shin and extended Schur bases respectively.  Combining Theorem \ref{back_eschur} and Proposition \ref{back_shin_omega} allows us to express the antipode on these bases in terms of the row-strict reverse extended Schur basis and the row-strict reverse shin basis.

\begin{cor} For a composition $\alpha$, $$S(\shin_{\alpha}) = (-1)^{|\alpha|} \bshin_{\alpha^r} \text{\quad and \quad} S^*(\shin^*_{\alpha}) = (-1)^{|\alpha|} \bshin^*_{\alpha^r}.$$
\end{cor}

This result reduces the problem to studying the expansion of the row-strict reverse extended Schur functions into the extended Schur functions and vice versa, which may be interesting to approach using tableaux combinatorics.

Finally, we introduce and study skew and \otherskew\ row-strict reverse extended Schur functions.

\begin{defn}
    For compositions $\beta \subseteq \alpha$, the \emph{skew row-strict reverse extended Schur functions} are defined by $$\bshin^*_{\alpha/\beta} = \bshin^{\perp}_{\beta}(\bshin^*_{\alpha}).$$ 
\end{defn}

$\bshin^*_{\alpha/\beta}$ expands into various bases according to Equation \eqref{NSym_perp}. For compositions $\beta \subseteq \alpha$, 
\begin{equation}\label{bshin_skew_exp}
     \bshin^*_{\alpha/\beta}= \sum_{\gamma} \langle \bshin_{\beta}H_{\gamma}, \bshin^*_{\alpha}  \rangle M_{\gamma} = \sum_{\gamma} \langle \bshin_{\beta}R_{\gamma}, \bshin^*_{\alpha} \rangle F_{\gamma} = \sum_{\gamma} \langle \bshin_{\beta}\bshin_{\gamma}, \bshin^*_{\alpha} \rangle \bshin^*_{\gamma}.
\end{equation}

Like before, the coefficients $\langle \shin_{\beta} \shin_{\gamma}, \shin^*_{\alpha}  \rangle$ also appear when multiplying row-strict reverse shin functions, $$\bshin_{\beta} \bshin_{\gamma} = \sum_{\alpha} \langle \bshin_{\beta} \bshin_{\gamma}, \bshin^*_{\alpha}  \rangle \bshin_{\alpha}.$$

Because $\omega$ is an anti-automorphism in \textit{NSym}, it does not map the skew extended Schur functions to the skew row-strict reverse extended Schur functions. Again, we use the right-perp operator to define functions that we show are the image of the skew extended Schur functions under $\omega$.

\begin{defn}\label{skewii_row-strict reverse}
    For compositions $\alpha$ and $\beta$ where $\beta^r \subseteq \alpha^r$, the \emph{\otherskew\  row-strict reverse extended Schur functions} are defined by $$\bshin^*_{\alpha\rskew \beta} = \bshin^{\rperp}_{\beta}(\bshin^*_{\alpha}).$$
\end{defn}

Using Definition \ref{NSym_rperp}, $\bshin^*_{\alpha \rskew \beta}$ expands into various bases as follows. For compositions $\alpha, \beta$ with $\beta^r \subseteq \alpha^r$, 
\begin{equation}\label{rskew_bshin_exp} \bshin^*_{\alpha \rskew \beta} = \sum_{\gamma} \langle H_{\gamma}\bshin_{\beta}, \bshin^*_{\alpha} \rangle M_{\gamma} = \sum_{\gamma} \langle R_{\gamma}\bshin_{\beta}, \bshin^*_{\alpha} \rangle F_{\gamma} = \sum_{\gamma} \langle \bshin_{\gamma} \bshin_{\beta}, \bshin^*_{\alpha} \rangle \bshin^*_{\gamma}.   
\end{equation}

\begin{thm}\label{omega_on_bshinskew}
    For compositions $\alpha$ and $\beta$ where $\beta \subseteq \alpha$, $$\omega(\shin^*_{\alpha/\beta}) = \bshin^*_{\alpha^r \rskew \beta^r}.$$
\end{thm}

\begin{proof}
First, observe that $$\langle \omega(\shin_{\beta} \shin_{\gamma}), \omega(\shin^*_{\alpha}) \rangle = \langle \omega(\shin_{\gamma}) \omega(\shin_{\beta}), \omega(\shin^*_{\alpha}) \rangle = \langle \bshin_{\gamma^r} \bshin_{\beta^r}, \bshin_{\alpha^r} \rangle,$$ because $\omega$ is invariant under duality and an anti-isomorphism in \textit{NSym}. Then, by Equation \eqref{rskew_bshin_exp}
$$\omega(\shin^*_{\alpha/\beta}) = \sum_{\gamma} \langle \shin_{\beta}\shin_{\gamma}, \shin^*_{\alpha} \rangle 
 \omega(\shin^*_{\gamma}) = \sum_{\gamma} \langle \shin_{\beta}\shin_{\gamma}, \shin^*_{\alpha} \rangle 
 \bshin^*_{\gamma^r} = \sum_{\gamma} \langle \bshin_{\gamma^r} \bshin_{\beta^r}, \bshin_{\alpha^r} \rangle \bshin^*_{\gamma^r} = \bshin^*_{\alpha^r \rskew \beta^r}.\qedhere$$ 
\end{proof}

The first line of the proof above shows the following equality.

\begin{cor}
    For compositions $\alpha, \beta$ and $\gamma$, 
    $$\bshin_{\gamma}\bshin_{\beta} = \sum_{\alpha} \mathcal{C}^{\alpha^r}_{\beta^r, \gamma^r} \bshin_{\alpha}.$$
\end{cor}

The \otherskew\ row-strict reverse extended Schur functions are again defined combinatorially via tableaux using \otherskew\ shapes.

\begin{defn}
    For compositions $\beta = (\beta_1, \ldots, b_{\ell}) \subseteq \alpha = (\alpha_1, \ldots, \alpha_k)$ such that $\beta^r \subseteq \alpha^r$, a \emph{row-strict reverse \otherskew\ shin-tableau} of skew shape $\alpha \rskew \beta$ is a \otherskew\ diagram $\alpha \rskew \beta$ filled with integers such that each row is strictly decreasing left to right, each column is weakly increasing top to bottom, and if $\alpha_{k-i}> \beta_{\ell-i}$ for some $i$ then there is no $j>i$ such that $\beta_{\ell-j}> \beta_{\ell - i}$. A \otherskew\  row-strict reverse shin-tableau is \emph{standard} if it contains the numbers $1$ through $|\alpha| - |\beta|$ each exactly once.

\end{defn}

\begin{ex}
   The three leftmost diagrams below are examples of  \otherskew\ row-strict reverse shin-tableaux while the rightmost diagram is not.

    $$ \text{\otherskew\ row-strict reverse shin:\quad}
    \scalebox{.85}
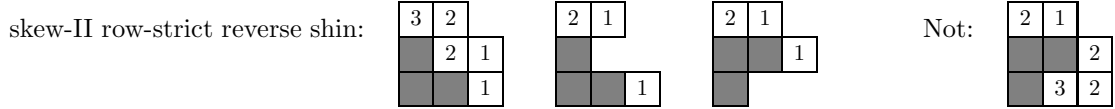
{
    \begin{ytableau}
          3 & 2 \\
        *(gray) & 2 & 1\\
        *(gray) &  *(gray)  & 1
    \end{ytableau}
    \qquad 
    \begin{ytableau}
        2 & 1 \\
        *(gray)\\ 
        *(gray) & *(gray) & 1 
    \end{ytableau}
    \qquad
    \begin{ytableau}
        2 & 1\\ 
         *(gray) & *(gray) & 1\\
        *(gray)   
    \end{ytableau}}
     \qquad
     \qquad
     \text{Not:\quad }
     \scalebox{.85}{
     \begin{ytableau}
        2 & 1\\
         *(gray) & *(gray) & 2 \\
        *(gray) & 3 & 2 
     \end{ytableau}}
    $$
\end{ex}

 The expansion of skew-II row-strict reverse shin functions in terms of tableaux closely follows that of Proposition \ref{rskew_rshin_tab}, but by instead applying $\psi$ to the \otherskew\ reverse extended Schur functions. 

\begin{prop}
   For compositions $\alpha$ and $\beta$ where $\beta^r \subseteq \alpha^r$, $$\bshin^*_{\alpha \rskew \beta} = \sum_{T} x^T,$$ where the sum runs over \otherskew\ reverse shin-tableaux $T$ of shape $\alpha \rskew \beta$.
\end{prop}

We also use skew-II row-strict reverse tableaux to give a combinatorial formula for the product of a ribbon function and a row-strict reverse shin function.  The proof exactly follows that of Proposition \ref{fshin_R_mult} but by applying $\rho$ to Theorem \ref{rshin_big_thm} (2).

\begin{prop}
Left Ribbon Multiplication. $$R_{\beta}\bshin_{\alpha} = \sum_{\gamma \models |\alpha| + |\beta|} \sum_{S} \bshin_{\gamma},$$ where the sum runs over \otherskew\ standard row-strict reverse shin-tableaux $S$ of shape $\gamma/\alpha$ with $co_{\bshin}(U) = \beta$. 
\end{prop}

\printbibliography

\end{document}